\documentclass[11pt]{article}
\usepackage{amssymb,amsmath}
\usepackage[mathscr]{eucal}
\usepackage[cm]{fullpage}
\usepackage[english]{babel}
\usepackage[latin1]{inputenc}
\usepackage{xcolor}
\usepackage{tikz}
\usepackage[scr=boondoxo]{mathalfa}
\usepackage{upgreek,float}

\title{On monotone alternating inverse monoids}

\author{
\textbf{V\'{\i}tor Hugo Fernandes}\\
Center for Mathematics and Applications (NOVA Math) 
and Department of Mathematics,  \\ 
Faculdade de Ci\^{e}ncias e Tecnologia, 
Universidade Nova de Lisboa, 
2829-516 Caparica, 
Portugal\\ 
E-mail: vhf@fct.unl.pt\\
ORCID iD: https://orcid.org/0000-0003-1057-4975\\ 
}

\date{}

\newtheorem{theorem}{Theorem}[section]
\newtheorem{corollary}[theorem]{Corollary}
\newtheorem{lemma}[theorem]{Lemma}
\newtheorem{proposition}[theorem]{Proposition}

\def\dom{\mathop{\mathrm{Dom}}\nolimits}
\def\im{\mathop{\mathrm{Im}}\nolimits}
\def\rank{\mathop{\mathrm{rank}}\nolimits} 
\def\con{\mathop{\mathrm{Con}}\nolimits} 
\def\gd{\mathrm{d}}
\def\gi{\mathrm{i}}
\def\id{\mathrm{id}}

\def\Z{\mathbb Z}
\def\Sym{\mathcal{S}}
\def\A{\mathcal{A}}
\def\POI{\mathcal{POI}}
\def\PMI{\mathcal{PMI}}
\def\I{\mathcal{I}}
\def\AI{\mathcal{AI}}
\def\AO{\mathcal{AO}} 
\def\AM{\mathcal{AM}} 
\newcommand{\transf}[1]{\left(\begin{smallmatrix} #1 \end{smallmatrix}\right)}
\renewcommand{\mod}[1]{\,(\mathrm{mod}{\,#1})}
\newcommand{\conpi}[1]{\uppi_{\mbox{$\!_{#1}$}}}
\newcommand{\contheta}[1]{\uptheta_{\mbox{$\!_{#1}$}}}
\newcommand{\rees}[1]{\sim_{\mbox{$\!_{#1}$}}}
\newenvironment{proof}{\begin{trivlist}\item[\hskip%
\labelsep{\bf\em Proof.}]}%
{\qed\rm\end{trivlist}}
\newcommand{\qed}{{\unskip\nobreak
\hfil\penalty50\hskip .001pt \hbox{}
          \nobreak\hfil
         {\scriptsize$\Box$}
           \parfillskip=0pt\finalhyphendemerits=0\medbreak}}

\begin{document}

\maketitle 

\begin{abstract} 
In this paper, we consider the inverse submonoids $\AM_n$ of monotone transformations and $\AO_n$ of order-preserving transformations 
of the alternating inverse monoid $\AI_n$ on a chain with $n$ elements. We compute the cardinalities, describe the Green's structures and the congruences, 
and calculate the ranks of these two submonoids of $\AI_n$. 
\end{abstract}

\noindent{\small\it Keywords: \rm alternating partial permutations, monotone, order-preserving, congruences, rank.}  

\medskip 

\noindent{\small 2020 \it Mathematics subject classification: \rm 20M20, 20M18, 20M10.}  

\section*{Introduction and Preliminaries}\label{Int} 

For a positive integer $n$, let $\Omega_n$ be a (finite) set with $n$ elements, e.g. $\Omega_n=\{1,2,\ldots,n\}$. 
We denote by $\Sym_n$  the \textit{symmetric group} on $\Omega_n$,
i.e. the group (under composition of mappings) of all permutations on $\Omega_n$, 
and by $\I_n$ the \textit{symmetric inverse monoid} on $\Omega_n$, i.e.
the inverse monoid (under composition of partial mappings) of all partial permutations (i.e. injective partial transformations) on $\Omega_n$.   
Transformation semigroups play a role in Semigroup Theory analogous to that of the symmetric group in Group Theory, while semigroups of partial permutations play a corresponding role in Inverse Semigroup Theory. 
This analogy is supported by two simple but fundamental results that parallel Cayley's theorem for groups. In Semigroup Theory, every semigroup is isomorphic to a subsemigroup of an appropriate full transformation semigroup. In Inverse Semigroup Theory, the Wagner--Preston Theorem states that every inverse semigroup is isomorphic to a subsemigroup of an appropriate symmetric inverse semigroup. 
In this paper, we focus on a particular class of partial permutations, namely the monotone injective partial transformations of a finite chain.

\smallskip 

Let $G$ be a subgroup of $\Sym_n$ and let $\I_n(G)=\{\alpha\in\I_n\mid \mbox{$\alpha=\sigma|_{\dom(\alpha)}$, for some $\sigma\in G$}\}$.  
It is clear that $\I_n(G)$ is an inverse submonoid of $\I_n$ containing the semilattice $\mathcal{E}_n$ of all idempotents of $\I_n$ and with $G$ as group of units. 
By taking $G=\Sym_n$, $G=\A_n$ or $G=\{\id_n\}$,
where $\A_n$ denotes the \textit{alternating group} on $\Omega_n$ and $\id_n$ is the identity transformation of $\Omega_n$, 
we obtain important and well-known inverse submonoids of $\I_n$. In fact, clearly,  
$\I_n(\Sym_n)=\I_n$ and $\I_n(\{\id_n\})=\mathcal{E}_n$. 
On the other hand, as we will show below, we have 
$\I_n(\A_n)=\AI_n$, the \textit{alternating semigroup}. 
Furthermore, the monoids $\I_n(\mathcal{C}_n)$ and $\I_n(\mathcal{D}_{2n})$, where $\mathcal{C}_n$ is a cyclic subgroup of $\Sym_n$ of order $n$ 
and $\mathcal{D}_{2n}$ is a dihedral subgroup of $\Sym_n$ of order $2n$,  
were studied in \cite{Fernandes:2024} and \cite{Fernandes&Paulista:2023}, respectively.  

\smallskip 

From now on, we will consider $\Omega_n$ as a chain,  e.g. $\Omega_n=\{1<2<\cdots<n\}$. 
In this context, an element $\alpha\in\I_n$ is called \textit{order-preserving}
[respectively, \textit{order-reversing}] if $x\leqslant y$ implies $x\alpha\leqslant y\alpha$
[respectively, $x\alpha\geqslant y\alpha$], for all $x,y \in \dom(\alpha)$.
A partial permutation is said to be \textit{monotone} if it is order-preserving or order-reversing.  
Let us denote by $\POI_n$
the inverse submonoid of $\I_n$ of all order-preserving
partial permutations and by $\PMI_n$ the inverse submonoid of $\I_n$
of all monotone partial permutations. 
Semigroups of monotone and of order-preserving transformations have been studied very extensively in recent decades. 
See, for example, 
\cite{Aizenstat:1962,Aizenstat:1962b,Cowan&Reilly:1995,Delgado&Fernandes:2000,Derech:1991,
Dimitrova&Koppitz:2009,Dimitrova&Koppitz:2023, 
Fernandes:1997,Fernandes:2000,Fernandes:2001,Fernandes:2024,Fernandes&Gomes&Jesus:2004,
Fernandes&Gomes&Jesus:2005,Fernandes&Paulista:2023,Fernandes&Volkov:2010,Ganyushkin&Mazorchuk:2003,
Garba:1994,Gomes&Howie:1992,Higgins:1995,Howie:1971,Laradji&Umar:2004,
Li&Fernandes:2024,Popova:1962,Schein&Teclezghi:1997,Vernitskii&Volkov:1995}.  

\smallskip 

In the present work, we consider the inverse submonoids $\AM_n=\AI_n\cap\PMI_n$ of monotone transformations and $\AO_n=\AI_n\cap\POI_n$ of order-preserving transformations of the alternating inverse monoid $\AI_n$. 
We aim to compute the cardinalities, describe the Green's structures and the congruences, and calculate the ranks of $\AM_n$ and $\AO_n$. 

\smallskip 

Recall that the \textit{rank} of a (finite) monoid $M$ is the minimum size of a generating set of $M$, i.e. 
the minimum of the set $\{|X|\mid \mbox{$X\subseteq M$ and $X$ generates $M$}\}$. 
As usual, we also use the term \textit{rank} relative to a transformation $\alpha$ of $\Omega_n$ to mean the number $|\im(\alpha)|$. 
Recall also that, for a monoid $M$, the Green's equivalence relations $\mathscr{L}$, $\mathscr{R}$, $\mathscr{J}$ and $\mathscr{H}$ are defined by
$a\mathscr{L}b$ if and only if $Ma=Mb$,  
$a\mathscr{R}b$ if and only if $aM=bM$,
$a\mathscr{J}b$ if and only if $MaM=MbM$ and 
$a\mathscr{H}b$ if and only if $a\mathscr{L}b$ and $a\mathscr{R}b$, for $a,b\in M$. 
In particular, given an inverse submonoid $M$ of $\I_n$,
it is well known that the Green's relations $\mathscr{L}$, $\mathscr{R}$ and $\mathscr{H}$ of $M$ can be described as follows:
for $\alpha, \beta \in M$,
\begin{itemize}
\item $\alpha \mathscr{L} \beta$ if and only if $\im(\alpha) = \im(\beta)$,

\item $\alpha \mathscr{R} \beta$ if and only if $\dom(\alpha) = \dom(\beta)$, and

\item $\alpha \mathscr{H} \beta$ if and only if $\im(\alpha) = \im(\beta)$ and $\dom(\alpha) = \dom(\beta)$.
\end{itemize}
Especially for $\I_n$, we also have
\begin{itemize}
\item $\alpha \mathscr{J} \beta$ if and only if $|\dom(\alpha)| = |\dom(\beta)|$ (if and only if $|\im(\alpha)| = |\im(\beta)|$).
\end{itemize}

Denote by $J_{a}$ the $\mathscr{J}$-class
of the element $a\in M$. As usual, a partial order relation
$\leqslant_\mathscr{J}$ is defined on the set $M/\mathscr{J}$ by
setting $J_{a}\leqslant_\mathscr{J}J_{b}$ if and only if 
$MaM\subseteq MbM$, for $a,b\in M$. We also 
write $J_{a}<_\mathscr{J}J_{b}$ if
and only if $J_a\leqslant_\mathscr{J}J_b$ and $(a, b)\not\in\mathscr{J}$, for $a, b\in M$. 
We denote by $E(M)$ the set of all idempotents of $M$  
and by $\leqslant$ the natural order on $E(M)$, i.e. its partial order relation defined  by
$e\leqslant f$ if and only if $ef=e=fe$, for $e, f\in E(M)$. 
An \textit{ideal} of $M$ is a subset $I$ of $M$ such that
$MIM\subseteq I$.  
The \textit{Rees congruence} of $M$ associated to an ideal $I$ of $M$ 
is the congruence $\rees{I}$ defined by $a\rees{I}b$ if and only if $a=b$ or
$a,b\in I$, for $a, b\in M$. 
We denote by $\id$ and $\omega$ the identity and universal congruences of $M$, respectively. 
The congruence \textit{lattice} (under the inclusion order relation) of $M$ is denoted by $\con(M)$. 

\medskip 

In this paper, the symbol $\subset$ stands for \textit{strictly} contained and, 
from now on, we will always consider $n\geqslant 2$.  

\smallskip 

Next, we recall the definition of the alternating semigroup $\AI_n$ (denoted by $\A^c_n$ in \cite{Lipscomb:1996}) 
and justify the equality $\I_n(\A_n)=\AI_n$ stated above. 

Let $J_k^{\I_n}=\{\alpha\in\I_n\mid |\im(\alpha)|=k\}$, for $0\leqslant k\leqslant n$. 
It is well known that 
$$
\I_n/\mathscr{J}=\{J_0^{\I_n}<_{\mathscr{J}}J_1^{\I_n}<_{\mathscr{J}}\cdots<_{\mathscr{J}}J_n^{\I_n}\}
$$ 
and, regarding the size of these $\mathscr{J}$-classes, we have $|J_k^{\I_n}|=\binom{n}{k}^2k!$, for $0\leqslant k\leqslant n$. 
Observe also that $J_0^{\I_n}=\{\emptyset\}$, 
where $\emptyset$ is the empty transformation of $\Omega_n$, 
and $J_n^{\I_n}=\Sym_n$. 

Let $\alpha\in J_{n-1}^{\I_n}$. Denote by $\overline{\alpha}$ the \textit{completion} of $\alpha$, i.e. $\overline{\alpha}$ is the unique permutation of 
$\Sym_n$ such that $\overline{\alpha}|_{\dom(\alpha)}=\alpha$. Let 
$J_{n-1}^{\A_n}=\{\alpha\in J_{n-1}^{\I_n}\mid \overline{\alpha}\in\A_n\}$. 
Notice that $|J_{n-1}^{\A_n}|=\frac{1}{2}n!n=\frac{1}{2}|J_{n-1}^{\I_n}|$. Then, we have 
$$
\AI_n=\A_n\cup J_{n-1}^{\A_n}\cup J_{n-2}^{\I_n}\cup\cdots\cup J_1^{\I_n}\cup J_0^{\I_n}
$$
(see \cite[Theorems 25.1 and 25.2]{Lipscomb:1996}) and so $|\AI_n|=\frac{1}{2}n!+\frac{1}{2}n!n+\sum_{k=0}^{n-2}\binom{n}{k}^2k!$. 

Now, observe that, clearly, $\I_n(\A_n)\cap J_n^{\I_n}=\A_n=\AI_n\cap J_n^{\I_n}$ and 
$\I_n(\A_n)\cap J_{n-1}^{\I_n}=J_{n-1}^{\A_n}=\AI_n\cap J_{n-1}^{\I_n}$. 
Let $\alpha\in J_k^{\I_n}$, for some $0\leqslant k\leqslant n-2$, and 
let $\sigma\in\Sym_n$ be any permutation such that $\sigma|_{\dom(\alpha)}=\alpha$. 
Let $i$ and $j$ be any two distinct elements of $\Omega_n\setminus\im(\alpha)$ and let $\tau$ be the transposition $(i\:j)$ of $\Sym_n$. 
Then, we also have $(\sigma\tau)|_{\dom(\alpha)}=\alpha$ and, since $\sigma\in\A_n$ if and only if $\sigma\tau\not\in\A_n$, we deduce that $\alpha\in\I_n(\A_n)$. 
Thus, $\I_n(\A_n)=\AI_n$. 

\medskip 

This paper is organized as follows. In Sections \ref{AO} and \ref{AM}, we define the structures of monoids $\AO_n$ and $\AM_n$, respectively. In Section \ref{con}, we describe their congruences. Finally, in Sections \ref{rAO} and \ref{rAM}, we calculate the ranks of $\AO_n$ and $\AM_n$, respectively.

\medskip 

We end this section by introducing some notation that we will use in subsequent sections.  

For $\alpha\in J_{n-1}^{\I_n}$, define $\gd(\alpha)$ and $\gi(\alpha)$ as the \textit{gaps} of $\dom(\alpha)$ and $\im(\alpha)$, respectively, 
i.e. $\{\gd(\alpha)\}=\Omega_n\setminus\dom(\alpha)$ and $\{\gi(\alpha)\}=\Omega_n\setminus\im(\alpha)$.  
For subsets $A$ and $B$ of $\Omega_n$ such that $|A|=|B|$, 
by $\transf{A\\B}$, we mean that $\alpha\in\I_n$ with $A=\dom(\alpha)$ and $B=\im(\alpha)$. 
Observe that, for any subsets $A$ and $B$ of $\Omega_n$ such that $|A|=|B|$, 
there exists a unique order-preserving [respectively, order-reversing] partial permutation $\alpha$ of $\Omega_n$ such that $\alpha=\transf{A\\B}$. 

\smallskip 

For general background on Semigroup Theory and standard notations, we refer to Howie's book \cite{Howie:1995}.

\smallskip 

We would like to point out that we made use of computational tools, namely GAP \cite{GAP4} and its package \cite{sgpviz}.

\section{The structure of $\AO_n$} \label{AO} 

Recall that $\POI_n/\mathscr{J}=\{J_0^{\POI_n}<_{\mathscr{J}}J_1^{\POI_n}<_{\mathscr{J}}\cdots<_{\mathscr{J}}J_n^{\POI_n}\}$, 
where $J_k^{\POI_n}=J_k^{\I_n}\cap\POI_n$, for $0\leqslant k\leqslant n$. 
In this case, we have $|J_k^{\POI_n}|=\binom{n}{k}^2$, for $0\leqslant k\leqslant n$, with 
$J_0^{\POI_n}=\{\emptyset\}$ and $J_n^{\POI_n}=\{\id_n\}$,  
see \cite{Fernandes:2001}. 

Obviously, the group of units $J_n=J_n^{\POI_n}\cap\AO_n$ (which is also a $\mathscr{J}$-class and a $\mathscr{H}$-class) of $\AO_n$ is trivial 
(as well as all its other $\mathscr{H}$-classes) and, on the other hand, we have 
$$
\{\alpha\in\AO_n\mid |\im(\alpha)|\leqslant n-2\}=\{\alpha\in\POI_n\mid |\im(\alpha)|\leqslant n-2\}=J_{n-2}^{\POI_n}\cup\cdots\cup J_1^{\POI_n}\cup J_0^{\POI_n}.
$$
Hence, it remains to characterize the elements of $\AO_n$ with rank $n-1$, which we will do next.

\begin{proposition}\label{chAO}
Let $\alpha\in J_{n-1}^{\POI_n}$. Then, $\alpha\in\AO_n$ if and only if $\gd(\alpha)$ and $\gi(\alpha)$ have the same parity. 
\end{proposition}
\begin{proof}
Let $i=\gd(\alpha)$ and $j=\gi(\alpha)$. 
Then, 
$
\Omega_n=\{1\alpha<\cdots<(i-1)\alpha<(i+1)\alpha<\cdots<n\alpha\}\cup\{j\}. 
$
So, let us consider the different cases depending on the position of $j$ relative to the image of $\alpha$. 

If $j<1\alpha$, then 
$$
j=1, 1\alpha=2,\ldots,(i-1)\alpha=i, (i+1)\alpha=i+1,\ldots,n\alpha=n, 
$$
whence $\overline\alpha=(1\:2\:\cdots\:i)$, 
and so $\alpha\in\AO_n$ if and only if $i-1$ is an even number. 

If $(k-1)\alpha<j<k\alpha$, for some $2\leqslant k\leqslant i-1$, then 
$$
1\alpha=1,\ldots,(k-1)\alpha=k-1, j=k, k\alpha=k+1,\ldots,(i-1)\alpha=i, (i+1)\alpha=i+1,\ldots,n\alpha=n, 
$$
whence $j<i$ and $\overline\alpha=(j\:j+1\:\cdots\:i)$, and so $\alpha\in\AO_n$ if and only if $i-j$ is an even number. 

If $(i-1)\alpha<j<(i+1)\alpha$, then 
$$
1\alpha=1,\ldots,(i-1)\alpha=i-1, j=i, (i+1)\alpha=i+1,\ldots,n\alpha=n, 
$$
whence $\overline\alpha=\id_n$, and so $\alpha\in\AO_n$. 

If $k\alpha<j<(k+1)\alpha$, for some $i+1\leqslant k\leqslant n-1$, then 
$$
1\alpha=1,\ldots,(i-1)\alpha=i-1, (i+1)\alpha=i,\ldots,k\alpha=k-1, j=k, (k+1)\alpha=k+1,\ldots,n\alpha=n, 
$$
whence $i<j$ and $\overline\alpha=(j\:j-1\:\cdots\:i)$, and so $\alpha\in\AO_n$ if and only if $j-i$ is an even number. 

Finally, if $n\alpha<j$, then 
$$
1\alpha=1,\ldots,(i-1)\alpha=i-1, (i+1)\alpha=i, \ldots, n\alpha=n-1, j\alpha=n,  
$$
whence $\overline\alpha=(n\:n-1\:\cdots\:i)$, 
and so $\alpha\in\AO_n$ if and only if $n-i$ is an even number. 

Thus, we can conclude that $\alpha\in\AO_n$ if and only if $i$ and $j$ have the same parity, as required. 
\end{proof} 

Since in $\Omega_n$ we have $\lfloor\frac{n}{2}\rfloor$ even numbers and $\lceil\frac{n}{2}\rceil$ odd numbers 
(i.e. $\frac{n}{2}$ even numbers and $\frac{n}{2}$ odd numbers, if $n$ is even, and 
$\frac{n-1}{2}$ even numbers and $\frac{n+1}{2}$ odd numbers, 
if $n$ is odd), then we obtain $2(\frac{n}{2})^2$ pairs with the same parity, if $n$ is even, 
and $(\frac{n-1}{2})^2+(\frac{n+1}{2})^2$ pairs with the same parity, if $n$ is odd. 
Therefore,
\begin{equation}\label{n-1o}
|\{\alpha\in\AO_n\mid |\im(\alpha)|=n-1\}|=\left\{
\begin{array}{ll}
\frac{1}{2}n^2 & \mbox{if $n$ is even}\\
\frac{n^2+1}{2} & \mbox{if $n$ is odd.}
\end{array}
\right. 
\end{equation}
and so, as $|J_{n-1}^{\POI_n}|=n^2$ and $|\POI_n|=\binom{2n}{n}$, we get the following result.

\begin{proposition}\label{carAO}
$|\AO_n|=\binom{2n}{n}-\lfloor\frac{1}{2}n^2\rfloor$.
\end{proposition}

For $0\leqslant k\leqslant n-2$, let us denote $J_k^{\POI_n}$ simply by $J_k$. It is easy to conclude that $J_0,J_1,\ldots,J_{n-2}$ 
are also $\mathscr{J}$-classes of $\AO_n$. Moreover, we have $J_0<_\mathscr{J}J_1<_\mathscr{J}\cdots<_\mathscr{J}J_{n-2}<_\mathscr{J}J_n$  in $\AO_n$. 
Regarding the elements of $\AO_n$ with rank $n-1$, we have the following result. 

\begin{proposition}\label{relJ}
Let $\alpha,\beta\in\AO_n$ be two elements of rank $n-1$. Then, $\alpha\mathscr{J}\beta$ if and only if $\gd(\alpha)$ and $\gd(\beta)$ have the same parity. 
\end{proposition}
\begin{proof}
First, let us suppose that $\alpha\mathscr{J}\beta$. Take $\gamma,\lambda\in\AO_n$ such that $\alpha=\gamma\beta\lambda$ and assume, 
without loss of generality, that $\gamma$ also has rank $n-1$. Then, $\dom(\alpha)=\dom(\gamma)$ and $\im(\gamma)=\dom(\beta)$,  
whence $\gd(\alpha)=\gd(\gamma)$ and $\gi(\gamma)=\gd(\beta)$.
Since $\gd(\gamma)$ and $\gi(\gamma)$ have the same parity, by Proposition \ref{chAO}, 
we obtain that $\gd(\alpha)$ and $\gd(\beta)$ have also the same parity. 

\smallskip 

On the other hand, suppose that $\gd(\alpha)$ and $\gd(\beta)$ have the same parity. 
Let $\gamma=\transf{\dom(\alpha)\\\dom(\beta)}\in\POI_n$ and $\lambda=\transf{\im(\beta)\\\im(\alpha)}\in\POI_n$. 
Then, clearly, $\alpha=\gamma\beta\lambda$ and both $\gamma$ and $\lambda$ have rank $n-1$. 
On the other hand, $\gd(\gamma)=\gd(\alpha)$, $\gi(\gamma)=\gd(\beta)$, $\gd(\lambda)=\gi(\beta)$, $\gi(\lambda)=\gi(\alpha)$ and, 
as $\gd(\alpha)$, $\gi(\alpha)$, $\gd(\beta)$ and $\gi(\beta)$ have the same parity, 
it follows that $\gd(\gamma),\gi(\gamma),\gd(\lambda)$ and $\gi(\lambda)$ have the same parity, 
whence $\gamma,\lambda\in\AO_n$,
by Proposition \ref{chAO}, 
and so $J_\alpha\leqslant_\mathscr{J}J_\beta$. 
Similarly, we can show that $J_\beta\leqslant_\mathscr{J}J_\alpha$ and thus $\alpha\mathscr{J}\beta$, as required. 
\end{proof} 

This last proposition enables us to conclude that $\AO_n$ has two $\mathscr{J}$-classes of elements of rank $n-1$, namely,
$$
J_{n-1}^\mathscr{o}=\{\alpha\in\AO_n\cap J_{n-1}^{\POI_n}\mid \mbox{$\gd(\alpha)$ is odd}\}
\quad\text{and}\quad 
J_{n-1}^\mathscr{e}=\{\alpha\in\AO_n\cap J_{n-1}^{\POI_n}\mid \mbox{$\gd(\alpha)$ is even}\}. 
$$
Clearly, $|J_{n-1}^\mathscr{o}|=|J_{n-1}^\mathscr{e}|=\frac{1}{4}n^2$, if $n$ is even, and 
$|J_{n-1}^\mathscr{o}|=\frac{(n+1)^2}{4}$ and $|J_{n-1}^\mathscr{e}|=\frac{(n-1)^2}{4}$, if $n$ is odd. 
Moreover, $J_{n-1}^\mathscr{o}$ has $\lceil\frac{n}{2}\rceil$ $\mathscr{L}$-classes (and $\mathscr{R}$-classes) and 
$J_{n-1}^\mathscr{e}$ has $\lfloor\frac{n}{2}\rfloor$  $\mathscr{L}$-classes (and $\mathscr{R}$-classes). 
Furthermore, it is easy to check that the poset of $\mathscr{J}$-classes of $\AO_n$ can be represented by the Hasse diagram of Figure \ref{fAO}. 
\begin{figure}[h] 
\centering
\begin{tikzpicture}[scale=0.5]
\draw (0,0) node{\fbox{$J_n$}} ; 
\draw (-2,-1.5) node{\fbox{$J_{n-1}^\mathscr{o}$}} ; 
\draw (2,-1.5) node{\fbox{$J_{n-1}^\mathscr{e}$}} ; 
\draw (0,-3) node{\fbox{$J_{n-2}$}} ; 
\draw (0,-5) node{\fbox{$J_1$}} ; 
\draw (0,-6.5) node{\fbox{$J_0$}} ; 
\draw[thin] (0,-.55) -- (-1,-0.89);
\draw[thin] (0,-.55) -- (1,-0.89);
\draw[thin] (0,-2.42) -- (1,-2.11);
\draw[thin] (0,-2.42) -- (-1,-2.11);
\draw[thin,dotted] (0,-3.65) -- (0,-4.35);
\draw[thin] (0,-5.55) -- (0,-5.95);
\end{tikzpicture}
\caption{The Hasse diagram of $\AO_n/_\mathscr{J}$.} \label{fAO}
\end{figure}
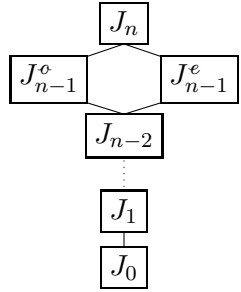

\section{The structure of $\AM_n$}\label{AM}

Recall that $\PMI_n/\mathscr{J}=\{J_0^{\PMI_n}<_{\mathscr{J}}J_1^{\PMI_n}<_{\mathscr{J}}\cdots<_{\mathscr{J}}J_n^{\PMI_n}\}$, 
where $J_k^{\PMI_n}=J_k^{\I_n}\cap\PMI_n$, for $0\leqslant k\leqslant n$. 
In this case, we have $|J_0^{\PMI_n}|=1$, $|J_1^{\PMI_n}|=n^2$ and $|J_k^{\PMI_n}|=2\binom{n}{k}^2$, for $2\leqslant k\leqslant n$. 
In particular,  $J_0^{\PMI_n}=\{\emptyset\}$ and $J_n^{\PMI_n}=\{\id_n,h\}$, with 
$$
h=\transf{1&2&\cdots&n-1&n\\n&n-1&\cdots&2&1} = \textstyle (1\:n)(2\:n)\cdots(\lfloor\frac{n}{2}\rfloor\:\lfloor\frac{n}{2}\rfloor+1), 
$$
see \cite{Fernandes&Gomes&Jesus:2004}. 

Observe that $h\in\AI_n$ if and only if $\lfloor\frac{n}{2}\rfloor$ is even if and only if $n\equiv 0\mod 4$ or $n\equiv 1\mod 4$. 

Let us abbreviate the expression ``$n\equiv i\mod 4$ or $n\equiv j\mod 4$" by ``$n\equiv \{i,j\}\mod 4$", for $i,j\in\Z$. 

Thus, for $Q_n=J_n^{\PMI_n}\cap\AM_n$ 
(i.e. the group of units of $\AM_n$ and so a $\mathscr{J}$-class of $\AM_n$), we get 
$$
Q_n=\left\{\begin{array}{ll}
\{\id_n,h\} & \mbox{if $n\equiv \{0,1\}\mod 4$}\\
\{\id_n\} & \mbox{if $n\equiv \{2,3\}\mod 4$} 
\end{array}\right. 
$$
and, on the other hand, we have 
$$
\{\alpha\in\AM_n\mid |\im(\alpha)|\leqslant n-2\}=\{\alpha\in\PMI_n\mid |\im(\alpha)|\leqslant n-2\}=J_{n-2}^{\PMI_n}\cup\cdots\cup J_1^{\PMI_n}\cup J_0^{\PMI_n}.
$$
Therefore, just like in $\AO_n$, it remains to characterize the elements of $\AM_n$ with rank $n-1$. 
In fact, as $\POI_n\subseteq\PMI_n$, it follows that $\AO_n\subseteq\AM_n$, whence we only need to characterize the elements
of $\AM_n\setminus\AO_n$ with rank $n-1$. 

First, we prove the following lemma. 

\begin{lemma}\label{ab}
Let $\alpha,\beta\in J_{n-1}^{\I_n}$ be such that $\alpha\beta\in J_{n-1}^{\I_n}$. Then, $\overline{\alpha\beta}=\overline\alpha\overline\beta$. 
\end{lemma}
\begin{proof}
Clearly, $\dom(\alpha\beta)=\dom(\alpha)$, $\im(\alpha)=\dom(\beta)$ and $\im(\alpha\beta)=\im(\beta)$. 
Then, if $x\in\dom(\alpha\beta)$, then $x\overline{\alpha\beta}=x(\alpha\beta)=(x\alpha)\beta=(x\overline\alpha)\overline\beta$. 
On the other hand, $\gd(\alpha) \overline{\alpha\beta} = \gd(\alpha\beta) \overline{\alpha\beta}=\gi(\beta) = \gd(\beta)\overline\beta = \gi(\alpha)\overline\beta=(\gd(\alpha)\overline\alpha)\overline\beta = \gd(\alpha)\overline\alpha\overline\beta$. Hence, $\overline{\alpha\beta}=\overline\alpha\overline\beta$, as required. 
\end{proof}

Next, observe that, if $i_1,i_2,\ldots,i_{n-1}\in\Omega_n$ and $\beta\in\PMI_n\setminus\POI_n$ are such that $i_1<i_2<\cdots<i_{n-1}$ and  
 $\dom(\beta)=\im(\beta)=\{i_1,i_2,\ldots,i_{n-1}\}$, then 
$
\overline\beta =  \textstyle (i_1\:i_{n-1})(i_2\:i_{n-1})\cdots(i_{\lfloor\frac{n-1}{2}\rfloor}\:i_{\lfloor\frac{n-1}{2}\rfloor+1})
$ 
and so $\overline\beta\in \AI_n$ if and only if $\lfloor\frac{n-1}{2}\rfloor$ is even if and only if $n\equiv \{1,2\}\mod 4$. 

\begin{proposition}\label{chAM}
Let $\alpha\in J_{n-1}^{\PMI_n}\setminus J_{n-1}^{\POI_n}$. Then:
\begin{enumerate}
\item For $n\equiv \{0,3\}\mod 4$,  $\alpha\in\AM_n$ if and only if $\gd(\alpha)$ and $\gi(\alpha)$ have distinct parities;
\item For $n\equiv \{1,2\}\mod 4$,  $\alpha\in\AM_n$ if and only if $\gd(\alpha)$ and $\gi(\alpha)$ have the same parity. 
\end{enumerate}
\end{proposition}
\begin{proof}
Let $\beta\in\PMI_n\setminus\POI_n$ be such that $\dom(\beta)=\im(\beta)=\im(\alpha)$. Then, $\alpha\beta\in J_{n-1}^{\I_n}$ and so, 
by Lemma \ref{ab}, $\overline{\alpha\beta}=\overline\alpha\overline\beta$. 
On the other hand, by the above observation, we can write $\overline\beta$ as a product of 
$\lfloor\frac{n-1}{2}\rfloor$ transpositions $\tau_1,\tau_2,\ldots,\tau_{\lfloor\frac{n-1}{2}\rfloor}$ of $\Omega_n$: 
$\overline\beta=\tau_1\tau_2\cdots\tau_{\lfloor\frac{n-1}{2}\rfloor}$. 
Suppose that $\overline{\alpha\beta}$ is a product of $p$ transpositions of $\Omega_n$. 
Then, we get $\overline\alpha=\overline{\alpha\beta}\tau_{\lfloor\frac{n-1}{2}\rfloor}\cdots\tau_2\tau_1$, and so, 
as $\dom(\alpha\beta)=\dom(\alpha)$, $\im(\alpha\beta)=\im(\beta)$ and $\alpha\beta\in\POI_n$, 
by Proposition \ref{chAO}, we have:

If $n\equiv \{1,2\}\mod 4$, then $\lfloor\frac{n-1}{2}\rfloor$ is even, whence $\alpha\in\AM_n$ if and only if $p+\lfloor\frac{n-1}{2}\rfloor$ is even, i.e.  
if and only if $p$ is even, i.e. if and only if $\alpha\beta\in\AO_n$, i.e. if and only if $\gd(\alpha\beta)$ and $\gi(\alpha\beta)$ have the same parity, i.e. 
if and only if $\gd(\alpha)$ and $\gi(\alpha)$ have the same parity; 

If $n\equiv \{0,3\}\mod 4$, then $\lfloor\frac{n-1}{2}\rfloor$ is odd, whence $\alpha\in\AM_n$ if and only if $p+\lfloor\frac{n-1}{2}\rfloor$ is even, i.e.  
if and only if $p$ is odd, i.e. if and only if $\alpha\beta\not\in\AO_n$, i.e. if and only if $\gd(\alpha\beta)$ and $\gi(\alpha\beta)$ have distinct parities, i.e. 
if and only if $\gd(\alpha)$ and $\gi(\alpha)$ have distinct parities. 
\end{proof} 

Let $\alpha\in\PMI_n$ be such that $|\im(\alpha)|\geqslant 2$. 
Denote by $\stackrel{\leftarrow}{\alpha}$ the transformation of $\PMI_n$ such that $\dom(\stackrel{\leftarrow}{\alpha})=\dom(\alpha)$, 
$\im(\stackrel{\leftarrow}{\alpha})=\im(\alpha)$ and $\stackrel{\leftarrow}{\alpha}\,\in\POI_n$ if and only if $\alpha\not\in\POI_n$. 
We say that $\stackrel{\leftarrow}{\alpha}$ is the \textit{reverse transformation} of $\alpha$. 
Then, as an immediate consequence of Propositions \ref{chAO} and \ref{chAM}, we have: 

\begin{corollary}\label{AM12}
Let $n\equiv \{1,2\}\mod 4$ and $\alpha\in J_{n-1}^{\PMI_n}$. Then, $\alpha\in\AM_n$ if and only if ${\stackrel{\leftarrow}{\alpha}} \in\AM_n$. 
\end{corollary}

\smallskip 

As already mentioned, with the same parity, we have $\frac{1}{2}n^2$ pairs, if $n$ is even, and $\frac{n^2+1}{2}$ pairs, if $n$ is odd. 
On the other hand, with distinct parities, we have $\frac{1}{2}n^2$ pairs, if $n$ is even, and $\frac{n^2-1}{2}$ pairs, if $n$ is odd. 
Therefore, 
\begin{equation*}
|\{\alpha\in\AM_n\cap(\PMI_n\setminus\POI_n)\mid |\im(\alpha)|=n-1\}|=\left\{
\begin{array}{ll}
\frac{1}{2}n^2 & \mbox{if $n$ is even}\\
\frac{n^2+1}{2} & \mbox{if $n\equiv 1\mod 4$}\\
\frac{n^2-1}{2} & \mbox{if $n\equiv 3\mod 4$}
\end{array}
\right. 
\end{equation*}
and, since $\{\alpha\in\AM_n\cap\POI_n\mid |\im(\alpha)|=n-1\}=\{\alpha\in\AO_n\mid |\im(\alpha)|=n-1\}$, 
by (\ref{n-1o}), we have 
\begin{equation*}
|\{\alpha\in\AM_n\cap\POI_n\mid |\im(\alpha)|=n-1\}|=\left\{
\begin{array}{ll}
\frac{1}{2}n^2 & \mbox{if $n$ is even}\\
\frac{n^2+1}{2} & \mbox{if $n$ is odd,}
\end{array}
\right. 
\end{equation*}
whence 
\begin{equation*} 
|\{\alpha\in\AM_n\mid |\im(\alpha)|=n-1\}|=\left\{
\begin{array}{ll}
n^2 & \mbox{if $n$ is even or $n\equiv 3\mod 4$}\\
n^2+1 & \mbox{if $n\equiv 1\mod 4$.}
\end{array}
\right. 
\end{equation*}
Hence, as $|J_{n-1}^{\PMI_n}|=2n^2$, $|\PMI_n|=2\binom{2n}{n}-n^2-1$ 
and $h\not\in\AM_n$ for $n\equiv \{2,3\}\mod 4$, 
we obtain the following result.

\begin{proposition}\label{carAM}
$|\AM_n|=2\binom{2n}{n}-2n^2-\left\{
\begin{array}{ll}
1& \mbox{if $n\equiv 0\mod 4$}\\
0 & \mbox{if $n\equiv 1\mod 4$}\\
2 & \mbox{if $n\equiv \{2,3\}\mod 4$.}
\end{array}
\right. $ 
\end{proposition}

\smallskip 

For $0\leqslant k\leqslant n-2$, let us denote $J_k^{\PMI_n}$ by $Q_k$. Then, as for $\AO_n$, 
it is easy to conclude that $Q_0,Q_1,\ldots,Q_{n-2}$ are also $\mathscr{J}$-classes of $\AM_n$ and 
$Q_0<_\mathscr{J}Q_1<_\mathscr{J}\cdots<_\mathscr{J}Q_{n-2}<_\mathscr{J}Q_n$  in $\AM_n$. 
With respect to the elements of $\AM_n$ with rank $n-1$, we have the following result. 

\begin{proposition}\label{MrelJ}
Let $\alpha,\beta\in\AM_n$ be two elements of rank $n-1$. Then:
\begin{enumerate}
\item For $n\equiv \{0,3\}\mod 4$,  $\alpha\mathscr{J}\beta$; 
\item For $n\equiv \{1,2\}\mod 4$,  $\alpha\mathscr{J}\beta$ if and only if $\gd(\alpha)$ and $\gd(\beta)$ have the same parity. 
\end{enumerate}
\end{proposition}
\begin{proof}
1. First, let us  suppose that $n\equiv \{0,3\}\mod 4$. 

Let $\gamma=\transf{\dom(\alpha)\\\dom(\beta)}, \lambda=\transf{\im(\beta)\\\im(\alpha)}, 
\gamma'=\transf{\dom(\beta)\\\dom(\alpha)},\lambda'=\transf{\im(\alpha)\\\im(\beta)}\in\I_n$. 
Then, $\gamma$, $\lambda$, $\gamma'$ and $\lambda'$ have rank $n-1$ with  
$\gd(\gamma)=\gd(\alpha)$, $\gi(\gamma)=\gd(\beta)$, $\gd(\lambda)=\gi(\beta)$, $\gi(\lambda)=\gi(\alpha)$, 
$\gd(\gamma')=\gd(\beta)$, $\gi(\gamma')=\gd(\alpha)$, $\gd(\lambda')=\gi(\alpha)$ and $\gi(\lambda')=\gi(\beta)$. 
Below, the transformations  $\gamma$, $\lambda$, $\gamma'$ and $\lambda'$ will be taken well-determined in $\POI_n$ or in $\PMI_n\setminus\POI_n$, depending on the different cases for $\alpha$ and $\beta$ that we are going to consider. 

\smallskip 

\noindent{\sc case} $\alpha,\beta\in\AO_n$ with $(-1)^{\gd(\alpha)}=(-1)^{\gd(\beta)}$, i.e $\gd(\alpha)$ and $\gd(\beta)$ have the same parity.  

In this case, by Proposition \ref{relJ}, we have $\alpha\mathscr{J}\beta$ in $\AO_n$, whence $\alpha\mathscr{J}\beta$ in $\AM_n$. 

\smallskip 

\noindent{\sc case} $\alpha,\beta\in\AO_n$ with $(-1)^{\gd(\alpha)}\neq(-1)^{\gd(\beta)}$, i.e $\gd(\alpha)$ and $\gd(\beta)$ have distinct parities.  

Take $\gamma,\lambda\in\PMI_n\setminus\POI_n$.  
Then, clearly, $\alpha=\gamma\beta\lambda$ and, 
since $(-1)^{\gi(\alpha)}=(-1)^{\gd(\alpha)}\neq(-1)^{\gd(\beta)}=(-1)^{\gi(\beta)}$, it follows that 
$(-1)^{\gd(\gamma)}\neq(-1)^{\gi(\gamma)}$ and $(-1)^{\gd(\lambda)}\neq(-1)^{\gi(\lambda)}$, 
whence $\gamma,\lambda\in\AM_n$, by Proposition \ref{chAM}, and so $J_\alpha\leqslant_\mathscr{J}J_\beta$. 
Similarly, we can show that $J_\beta\leqslant_\mathscr{J}J_\alpha$ and thus $\alpha\mathscr{J}\beta$. 

\smallskip 

\noindent{\sc case} $\alpha\in\AO_n$ and $\beta\in\AM_n\setminus\AO_n$.  

Then, we have $(-1)^{\gd(\alpha)}=(-1)^{\gi(\alpha)}$ and  $(-1)^{\gd(\beta)}\neq(-1)^{\gi(\beta)}$.  
If $(-1)^{\gd(\alpha)}=(-1)^{\gi(\alpha)}=(-1)^{\gd(\beta)}$, then take $\gamma,\gamma'\in\POI_n$ and $\lambda,\lambda'\in\PMI_n\setminus\POI_n$. 
On the other hand, if $(-1)^{\gd(\alpha)}=(-1)^{\gi(\alpha)}=(-1)^{\gi(\beta)}$, then take $\gamma,\gamma'\in\PMI_n\setminus\POI_n$ and $\lambda,\lambda'\in\POI_n$. 
Thus, in both cases, it is a routine matter to show that $\alpha=\gamma\beta\lambda$, $\beta=\gamma'\alpha\lambda'$ and $\gamma,\lambda,\gamma',\lambda'\in\AM_n$, 
whence $\alpha\mathscr{J}\beta$. 

\smallskip 

\noindent{\sc case} $\alpha,\beta\in\AM_n\setminus\AO_n$.  

Then, we have $(-1)^{\gd(\alpha)}\neq(-1)^{\gi(\alpha)}$ and  $(-1)^{\gd(\beta)}\neq(-1)^{\gi(\beta)}$.  
If $(-1)^{\gd(\alpha)}=(-1)^{\gd(\beta)}$, then $(-1)^{\gi(\alpha)}=(-1)^{\gi(\beta)}$ and so take $\gamma,\lambda,\gamma',\lambda'\in\POI_n$.  
On the other hand, if $(-1)^{\gd(\alpha)}\neq(-1)^{\gd(\beta)}$, then $(-1)^{\gi(\alpha)}\neq(-1)^{\gi(\beta)}$ and, in this case, 
take $\gamma,\lambda,\gamma',\lambda'\in\PMI_n\setminus\POI_n$.  
Thus, in both cases, it is easy to show that $\alpha=\gamma\beta\lambda$, $\beta=\gamma'\alpha\lambda'$ and $\gamma,\lambda,\gamma',\lambda'\in\AM_n$, 
whence $\alpha\mathscr{J}\beta$. 

\smallskip 

2. Now, suppose that $n\equiv \{1,2\}\mod 4$.   

Let us first admit that $\alpha\mathscr{J}\beta$. Take $\gamma,\lambda\in\AM_n$ such that $\alpha=\gamma\beta\lambda$ and assume, 
without loss of generality, that $\gamma$ has also rank $n-1$. Then, $\dom(\alpha)=\dom(\gamma)$ and $\im(\gamma)=\dom(\beta)$,  
whence $\gd(\alpha)=\gd(\gamma)$ and $\gi(\gamma)=\gd(\beta)$.
Since $\gd(\gamma)$ and $\gi(\gamma)$ have the same parity, by Propositions \ref{chAO} and \ref{chAM}(2), 
we get that $\gd(\alpha)$ and $\gd(\beta)$ have also the same parity. 

\smallskip 

On the other hand, suppose that $\gd(\alpha)$ and $\gd(\beta)$ have the same parity. 
Let $\gamma=\transf{\dom(\alpha)\\\dom(\beta)}\in\POI_n$.  
If both $\alpha$ and $\beta$ belong to $\AO_n$ or both belong to $\AM_n\setminus\AO_n$, then let $\lambda=\transf{\im(\beta)\\\im(\alpha)}\in\POI_n$. 
Otherwise, let $\lambda=\transf{\im(\beta)\\\im(\alpha)}\in\PMI_n\setminus\POI_n$. 
Then, clearly, $\alpha=\gamma\beta\lambda$ and both $\gamma$ and $\lambda$ have rank $n-1$. 
On the other hand, $\gd(\gamma)=\gd(\alpha)$, $\gi(\gamma)=\gd(\beta)$, $\gd(\lambda)=\gi(\beta)$, $\gi(\lambda)=\gi(\alpha)$ and, 
as $\gd(\alpha)$, $\gi(\alpha)$, $\gd(\beta)$ and $\gi(\beta)$ have the same parity, 
it follows that $\gd(\gamma),\gi(\gamma),\gd(\lambda)$ and $\gi(\lambda)$ have the same parity, 
whence $\gamma,\lambda\in\AM_n$,
by Propositions \ref{chAO} and \ref{chAM}(2), 
and so $J_\alpha\leqslant_\mathscr{J}J_\beta$. 
Similarly, we can show that $J_\beta\leqslant_\mathscr{J}J_\alpha$ and thus $\alpha\mathscr{J}\beta$, as required. 
\end{proof} 

Let $n\equiv \{0,3\}\mod 4$. Then, by the previous proposition, 
$$
Q_{n-1}=\{\alpha\in\AM_n\mid |\im(\alpha)|=n-1\}
$$
is a $\mathscr{J}$-class of $\AM_n$. Moreover, $|Q_{n-1}|=n^2$ and 
it is not difficult to deduce that $Q_{n-1}$ has $n$ $\mathscr{L}$-classes (and $\mathscr{R}$-classes), 
whence $Q_{n-1}$  has $n^2$ $\mathscr{H}$-classes,  which must therefore be trivial. 
Furthermore, we get $Q_0<_\mathscr{J}Q_1<_\mathscr{J}\cdots<_\mathscr{J}Q_{n-2}<_\mathscr{J}Q_{n-1}<_\mathscr{J}Q_n$ in $\AM_n$. 

\smallskip 

On the other hand, let $n\equiv \{1,2\}\mod 4$. In this case, Proposition \ref{MrelJ} 
allows us to conclude that $\AM_n$ has two $\mathscr{J}$-classes of elements of rank $n-1$, namely,
$$
Q_{n-1}^\mathscr{o}=\{\alpha\in\AM_n\cap J_{n-1}^{\PMI_n}\mid \mbox{$\gd(\alpha)$ is odd}\}
\quad\text{and}\quad 
Q_{n-1}^\mathscr{e}=\{\alpha\in\AM_n\cap J_{n-1}^{\PMI_n}\mid \mbox{$\gd(\alpha)$ is even}\}. 
$$
Moreover, $|Q_{n-1}^\mathscr{o}|=|Q_{n-1}^\mathscr{e}|=\frac{1}{2}n^2$, if $n\equiv 2\mod 4$, 
$|Q_{n-1}^\mathscr{o}|=\frac{(n+1)^2}{2}$ and $|Q_{n-1}^\mathscr{e}|=\frac{(n-1)^2}{2}$, if $n\equiv 1\mod 4$,  
and it is not difficult to conclude that $Q_{n-1}^\mathscr{o}$ has $\lceil\frac{n}{2}\rceil$ $\mathscr{L}$-classes (and $\mathscr{R}$-classes),  
$Q_{n-1}^\mathscr{e}$ has $\lfloor\frac{n}{2}\rfloor$  $\mathscr{L}$-classes (and $\mathscr{R}$-classes) and both 
$\mathscr{J}$-classes have $\mathscr{H}$-classes with two elements.  
In addition, it is also easy to check that the poset of $\mathscr{J}$-classes of $\AM_n$ can be represented by the Hasse diagram of Figure \ref{fAM}. 
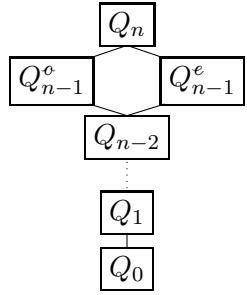
\begin{figure}[h] 
\centering
\begin{tikzpicture}[scale=0.5]
\draw (0,0) node{\fbox{$Q_n$}} ; 
\draw (-2,-1.5) node{\fbox{$Q_{n-1}^\mathscr{o}$}} ; 
\draw (2,-1.5) node{\fbox{$Q_{n-1}^\mathscr{e}$}} ; 
\draw (0,-3) node{\fbox{$Q_{n-2}$}} ; 
\draw (0,-5) node{\fbox{$Q_1$}} ; 
\draw (0,-6.5) node{\fbox{$Q_0$}} ; 
\draw[thin] (0,-.55) -- (-1,-0.89);
\draw[thin] (0,-.55) -- (1,-0.89);
\draw[thin] (0,-2.42) -- (1,-2.11);
\draw[thin] (0,-2.42) -- (-1,-2.11);
\draw[thin,dotted] (0,-3.65) -- (0,-4.35);
\draw[thin] (0,-5.55) -- (0,-5.95);
\end{tikzpicture}
\caption{The Hasse diagram of $\AM_n/_\mathscr{J}$, for $n\equiv \{1,2\}\mod 4$.} \label{fAM}
\end{figure}

\section{The congruences}\label{con}

Let $M$ be a finite $\mathscr{H}$-trivial (\textit{aperiodic}) inverse monoid. Then, $M$ has zero, which we denote by $0$. 
Observe that, $\{0\}$ and $M$ are ideals of $M$, and we have $\rees{\{0\}}=\id$ and $\rees{M}=\omega$. 

In order to describe the congruences of the monoid $\AO_n$, we begin by proving the following lemma.

\begin{lemma}\label{rfund}
Let $M$ be a finite $\mathscr{H}$-trivial inverse monoid. 
Let $\rho$ be a congruence of $M$ such that, for any $e\in E(M)$, if there exists $f\in E(M)$ such that $f<e$ and $e\rho f$, then $e\rho0$. 
Then, $\rho$ is a Rees congruence of $M$. 
\end{lemma}
\begin{proof}
We begin this proof by showing that, for $a\in M$, 
\begin{equation}\label{rfund1}
\mbox{$|a\rho|>1$ implies $a\rho0$.}  
\end{equation} 
Let $a\in M$ be such that there exists $b\in a\rho\setminus\{a\}$.  

First, suppose that $a,b\in E(M)$. Then, $a=a^2\rho ab$ and $ab\in E(M)$. Since $ab\leqslant a$, if $ab\neq a$, then $a\rho0$, by hypothesis; otherwise $b<a$ and, again by hypothesis, we get $a\rho0$. 

Next, we consider the general case. Since $a\rho b$ (and $M$ is an inverse monoid), then $a^{-1}\rho b^{-1}$, whence $aa^{-1}\rho bb^{-1}$ and $a^{-1}a\rho b^{-1}b$. 
If $aa^{-1} = bb^{-1}$ and $a^{-1}a = b^{-1}b$, then $a\mathscr{H}b$ and so, by hypothesis, $a=b$, which is a contradiction. Hence, 
$aa^{-1} \neq bb^{-1}$ or $a^{-1}a \neq b^{-1}b$. Therefore, by the idempotent case, we get $aa^{-1}\rho0$ or $a^{-1}a\rho0$ and so   
$a=aa^{-1}a\rho0a=0$ or $a=aa^{-1}a\rho a0=0$, whence $a\rho0$. Thus, we proved (\ref{rfund1}).

\smallskip 

Now, let $X_\rho=\left\{J\in M/\mathscr{J}\mid \mbox{$|a\rho|>1$, for some $a\in J$}\right\}$. 
If $X_\rho=\emptyset$, then $\rho=\id$ and so $\rho$ is the Rees congruence associated to the \textit{null ideal} $\{0\}$ of $M$. 
So, suppose that $X_\rho\neq\emptyset$. 

Let $J_{a_1},\ldots,J_{a_k}$ be (all) the $\leqslant_\mathscr{J}$-maximal elements of $X_\rho$, where $a_1,\ldots,a_k\in M$ are such that $|a_i\rho|>1$, for $i=1,\ldots,k$. 
Let $I=\left\{a\in M\mid\mbox{$J_a\leqslant_\mathscr{J}J_{a_i}$, for some $i=1,\ldots,k$}\right\}$. 
It is easy to show that $I$ is an ideal of $M$. 

Let $a\in I$. Then, $J_a\leqslant_\mathscr{J}J_{a_i}$, for some $i=1,\ldots,k$, and so $a=ua_iv$, for some $u,v\in M$. Since $|a_i\rho|>1$, by (\ref{rfund1}), we have $a_i\rho0$, 
whence $a=ua_iv\rho u0v=0$. 
Next, let $a\in M\setminus I$. If $|a\rho|>1$, then $J_a\in X_\rho$ and so $J_a\leqslant_\mathscr{J}J_{a_i}$, for some $i=1,\ldots,k$, 
since $J_{a_1},\ldots,J_{a_k}$ are the $\leqslant_\mathscr{J}$-maximal elements of $X_\rho$, whence $a\in I$, which is a contradiction. 
Hence, $|a\rho|=1$. 
Thus, $\rho$ is the Rees congruence associated to the ideal $I$ of $M$, as required. 
\end{proof}

For a nonempty subset $X$ of $\Omega_n$, let us denote by $\id_X$ the partial identity $\id_n|_X$. 
Before our next lemma, observe that $E(\AO_n)=E(\I_n)=\mathcal{E}_n=\{\id_X\mid X\subseteq\Omega_n\}$ and, 
given $\varepsilon,\zeta\in\mathcal{E}_n$,  we have $\zeta\leqslant\varepsilon$ if and only if $\dom(\zeta)\subseteq\dom(\varepsilon)$. 

\begin{lemma}\label{erho0}
Let $\rho$ be a congruence of $\AO_n$ and let $\varepsilon\in E(\AO_n)$. If there exists $\zeta\in E(\AO_n)$ such that $\zeta<\varepsilon$ and $\varepsilon\rho\zeta$, 
then $\varepsilon\rho\emptyset$. 
\end{lemma}
\begin{proof}
Let $\zeta\in E(\AO_n)$ be such that $\zeta<\varepsilon$ and $\varepsilon\rho\zeta$. 
Let $\gamma\in E(\AO_n)$ be such that $\gamma\leqslant\varepsilon$ and $\gamma\rho\varepsilon$, with $|\dom(\gamma)|$ minimum among all elements that verify these conditions. Then, $|\dom(\gamma)|\leqslant|\dom(\zeta)|<|\dom(\varepsilon)|$, whence $\dom(\gamma)\subset\dom(\varepsilon)$ and so we may take
 $x\in\dom(\varepsilon)\setminus\dom(\gamma)$. 

Suppose that $\gamma\neq\emptyset$. Take $y\in\dom(\gamma)$ and $X=(\dom(\gamma)\setminus\{y\})\cup\{x\}$. 
Additionally, if $|\dom(\gamma)|=n-1$ (whence $\varepsilon=\id_n$), then take $y\in\dom(\gamma)$ such that $(-1)^y=(-1)^x$. 
So, $|X|=|\dom(\gamma)|$. 
Let $\alpha$ be the only element of $\POI_n$ such that $\dom(\alpha)=\dom(\gamma)$ and $\im(\alpha)=X$. 
If $|\dom(\gamma)|<n-1$, then trivially $\alpha\in\AO_n$. On the other hand, if $|\dom(\gamma)|=n-1$, then $\gd(\alpha)=x$ and $\gi(\alpha)=y$, whence 
$(-1)^{\gd(\alpha)}=(-1)^x=(-1)^y=(-1)^{\gi(\alpha)}$ and so $\alpha\in\AO_n$.  

Since $X\subseteq\dom(\varepsilon)$, we have $\alpha\varepsilon\alpha^{-1}=\gamma$. 
On the other hand, 
$\alpha\gamma\alpha^{-1}$ is an idempotent of $\AO_n$ such that  $\varepsilon\rho\gamma=\alpha\varepsilon\alpha^{-1}\rho\alpha\gamma\alpha^{-1}$ and 
$\dom(\alpha\gamma\alpha^{-1})=\dom(\gamma)\setminus\{x\alpha^{-1}\}\subset\dom(\gamma)\subset\dom(\varepsilon)$. 
Hence, we have $\alpha\gamma\alpha^{-1}\leqslant\varepsilon$, $\alpha\gamma\alpha^{-1}\rho\varepsilon$ and $|\dom(\alpha\gamma\alpha^{-1})|<|\dom(\gamma)|$, 
which is a contradiction. 

Thus, $\gamma=\emptyset$ and so $\varepsilon\rho\emptyset$, as required.
\end{proof}

\smallskip 

Let $M$ be a monoid and let $I$ be an ideal of $M$. If $u\in I$ and $a\in M$ are such that $J_a\leqslant_\mathscr{J} J_u$, 
then $a\in MuM$ and so $a\in I$. Therefore, if $M/\mathscr{J} =\{J_{a_0}<_\mathscr{J} J_{a_1}<_\mathscr{J}\cdots<_\mathscr{J} J_{a_m}\}$, 
for some $a_0,a_1,\ldots,a_m\in M$, it is easy to conclude that the ideals of $M$ are the following $m+1$ subsets of $M$: 
$$
I_k=J_{a_0}\cup J_{a_1}\cup\cdots\cup J_{a_k}, ~ 0\leqslant k\leqslant m. 
$$
On the other hand, if 
$M/\mathscr{J} =\{J_{a_0}<_\mathscr{J} J_{a_1}<_\mathscr{J}\cdots<_\mathscr{J} J_{a_{m-2}}<_\mathscr{J} J_{a'_{m-1}}, J_{a''_{m-1}}<_\mathscr{J} J_{a_{m}}\}$ 
(with $J_{a'_{m-1}}$ and $J_{a''_{m-1}}$ being $\leqslant_\mathscr{J}$-incomparable), 
for some $a_0,a_1,\ldots,a_{m-2},a'_{m-1},a''_{m-1},a_m\in M$, it is not difficult to deduce that the ideals of $M$ are the following $m+3$ subsets of $M$: 
$$
I_k=J_{a_0}\cup J_{a_1}\cup\cdots\cup J_{a_k}, ~ 0\leqslant k\leqslant m-2, 
$$ 
$I'_{m-1}=I_{m-2}\cup J_{a'_{m-1}}$, $I''_{m-1}=I_{m-2}\cup J_{a''_{m-1}}$, $I_{m-1}=I_{m-2}\cup J_{a'_{m-1}}\cup J_{a''_{m-1}}$ and $I_m=M$. 

\medskip 

In particular, since 
$$
\AO_n/\mathscr{J}=\left\{J_0<_\mathscr{J}J_1<_\mathscr{J}\cdots<_\mathscr{J}J_{n-2}<_\mathscr{J} J_{n-1}^\mathscr{o}, J_{n-1}^\mathscr{e} <_\mathscr{J}J_n\right\}
$$ 
(see Figure \ref{fAO}), then $\AO_n$ has the following $n+3$ ideals:
$$
I_k=\left\{\alpha\in\AO_n\mid |\im(\alpha)|\leqslant k\right\}, ~ 0\leqslant k\leqslant n-2, 
$$
$I^\mathscr{o}_{n-1}=I_{n-2}\cup J_{n-1}^\mathscr{o}$, 
$I^\mathscr{e}_{n-1}=I_{n-2}\cup J_{n-1}^\mathscr{e}$, 
$$
I_{n-1}=I_{n-2}\cup J_{n-1}^\mathscr{o}\cup J_{n-1}^\mathscr{e} = \left\{\alpha\in\AO_n\mid |\im(\alpha)|\leqslant n-1\right\}
$$ 
and $I_n=\AO_n$. 

\smallskip 

Therefore, as an immediate consequence of Lemmas \ref{rfund} and \ref{erho0}, we have: 

\begin{theorem} \label{conAO} 
The congruences of the monoid $\AO_n$ are exactly its $n+3$ Rees congruences. 
\end{theorem} 

Observe that, given a monoid $M$ and two ideals $I$ and $I'$ of $M$, we have $\rees{I}\subseteq\rees{I'}$ if and only if $I\subseteq I'$. 
Therefore, it is clear that the lattice $\con(\AO_n)$ can be represented by the Hasse diagram of Figure \ref{cAO}. 
\begin{figure}[H] 
\centering
\begin{tikzpicture}[scale=0.5]
\draw (0,1.4) node{\fbox{$\omega$}} ; 
\draw (0,-.09) node{\fbox{$\sim_{I_{n-1}}$}} ; 
\draw (-2,-1.62) node{\fbox{$\rees{I_{n-1}^\mathscr{o}}$}} ; 
\draw (2,-1.62) node{\fbox{$\rees{I_{n-1}^\mathscr{e}}$}} ; 
\draw (0,-3.1) node{\fbox{$\rees{I_{n-2}}$}} ; 
\draw (0,-4.98) node{\fbox{$\rees{I_1}$}} ; 
\draw (0,-6.4) node{\fbox{\scriptsize$\id$}} ; 
\draw[thin] (0,-.6) -- (-1,-1.02);
\draw[thin] (0,-.6) -- (1,-1.02);
\draw[thin] (0,-2.52) -- (1,-2.22);
\draw[thin] (0,-2.52) -- (-1,-2.22);
\draw[thin,dotted] (0,-3.7) -- (0,-4.4);
\draw[thin] (0,-5.5) -- (0,-5.98);
\draw[thin] (0,0.4) -- (0,1);
\end{tikzpicture}
\caption{The Hasse diagram of $\con(\AO_n)$.} \label{cAO}
\end{figure}
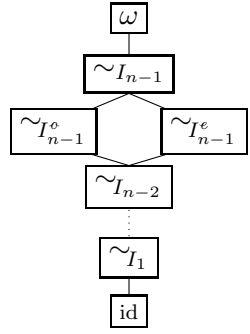

\medskip 

We now proceed to describe the congruences of the monoid $\AM_n$. 

\smallskip 

Let $M$ be a finite monoid and let $J$ be a $\mathscr{J}$-class
of $M$. Define 
$$
A(J)=\left\{ a\in M\mid J_a<_\mathscr{J}J \right\}
\quad\text{and}\quad 
B(J)=\left\{ a\in M\mid J\not\leqslant_\mathscr{J}J_a \right\}.
$$
It is not difficult to show that $A(J)$ and $B(J)$ 
are ideals of $M$ such that $A(J)\subseteq B(J)$ 
(more precisely, 
$
B(J)=A(J)\cup \left\{ a\in M\mid \mbox{$J$ and $J_a$ are $\leqslant_\mathscr{J}$-incomparable}\right\}
$) and $J\cap A(J)=J\cap B(J)=\emptyset$. 
Let $\contheta{J}$ [respectively, $\conpi{J}$] be the binary relation on $M$ defined by:
for $a,b\in M$, $a\contheta{J} b$ [respectively, $a\conpi{J} b$] if and only if
\begin{enumerate}
\item $a=b$; or
\item $a,b\in A(J)$ [respectively, $a,b\in B(J)$]; or
\item $a,b\in J$ and $a\mathscr{H}b$.
\end{enumerate}

In \cite[Lemma 4.2]{Fernandes:2000}, Fernandes proved that $\conpi{J}$ is a congruence of $M$. 
Next, we show that $\contheta{J}$ is also a congruence of $M$. Although the proof is very similar to the result mentioned above, 
we present it for the sake of completeness. 

\begin{lemma} \label{theta}
Let $M$ be a finite monoid and let $J$ be a $\mathscr{J}$-class
of $M$. Then, the relation $\contheta{J}$  is a congruence of $M$ such that $\contheta{J}\subseteq\conpi{J}$. 
\end{lemma}
\begin{proof}
Since $A(J)\subseteq B(J)$, we immediately get $\contheta{J}\subseteq\conpi{J}$. 

It is clear that $\contheta{J}$ is an equivalence relation on $M$.
Therefore, it remains to prove that $\contheta{J}$ is compatible with multiplication.
Let $a,b\in M$ be such that $a\contheta{J}b$ and let $u\in M$. 
If $a=b$ then $ua=ub$ and so $ua\contheta{J}ub$. 
If $a,b\in A(J)$, as $A(J)$ is an ideal of $M$, then
$ua,ub\in A(J)$, whence $ua\contheta{J}ub$. 
So, suppose that $a,b\in J$ and $a\mathscr{H}b$. 
Then, in particular, $a\mathscr{R}b$ and so $ua\mathscr{R}ub$, 
from which follows that $ua\mathscr{J}ub$. 
Hence, we have $ua,ub\in A(J)$ or $ua,ub\in J$, since $J_{ua},J_{ub}\leqslant_\mathscr{J}J_a=J_b=J$. 
If $ua,ub\in A(J)$, then $ua\contheta{J}ub$. 
So, suppose that $ua,ub\in J$. Since $M$ is a finite semigroup, from $a\mathscr{J}ua$, we get 
$a\mathscr{L}ua$. Similarly, we obtain $b\mathscr{L}ub$. Since we also have $a\mathscr{L}b$, it follows that 
$ua\mathscr{L}ub$ and so $ua\mathscr{H}ub$. Hence, $ua\contheta{J}ub$. 
Similarly, we can prove that $\contheta{J}$ is compatible with multiplication on the right, as required. 
\end{proof}

The following lemma is easy to prove. 

\begin{lemma} \label{union}
Let $M$ be a finite monoid and let $J$ and $J'$ be two $\mathscr{J}$-classes 
of $M$ such that $A(J)=A(J')$. 
Then, the relation $\contheta{J}\cup\contheta{{J'}}$ is a congruence of $M$. 
\end{lemma}

Observe that, to show the previous lemma, the only non-trivial property we need to prove is transitivity. 
Observe also that, if $J$ and $J'$ are two distinct $\mathscr{J}$-classes 
of a monoid $M$ such that $A(J)=A(J')$, then $J$ and $J'$ are $\leqslant_\mathscr{J}$-incomparable. 

\medskip 

Next, recall that, for $n\equiv \{0,3\}\mod 4$, we have 
$$
\AM_n/\mathscr{J}=\left\{Q_0<_\mathscr{J}Q_1<_\mathscr{J}\cdots<_\mathscr{J}Q_{n-2}<_\mathscr{J} Q_{n-1} <_\mathscr{J}Q_n\right\}
$$ 
and so, in this case, $\AM_n$ has the following $n+1$ ideals:
$$
F_k=\left\{\alpha\in\AM_n\mid |\im(\alpha)|\leqslant k\right\}, ~ 0\leqslant k\leqslant n. 
$$
On the other hand, for $n\equiv \{1,2\}\mod 4$, we have 
$$
\AM_n/\mathscr{J}=\left\{Q_0<_\mathscr{J}Q_1<_\mathscr{J}\cdots<_\mathscr{J}Q_{n-2}<_\mathscr{J} Q_{n-1}^\mathscr{o}, Q_{n-1}^\mathscr{e} <_\mathscr{J}Q_n\right\}
$$ 
(see Figure \ref{fAM}) and so, in this case, $\AM_n$ has the following $n+3$ ideals:
$$
F_k=\left\{\alpha\in\AM_n\mid |\im(\alpha)|\leqslant k\right\}, ~ 0\leqslant k\leqslant n-2, 
$$
$F^\mathscr{o}_{n-1}=F_{n-2}\cup Q_{n-1}^\mathscr{o}$, 
$F^\mathscr{e}_{n-1}=F_{n-2}\cup Q_{n-1}^\mathscr{e}$, 
$$
F_{n-1}=F_{n-2}\cup Q_{n-1}^\mathscr{o}\cup Q_{n-1}^\mathscr{e} = \left\{\alpha\in\AM_n\mid |\im(\alpha)|\leqslant n-1\right\}
$$ 
and $F_n=\AM_n$. 

\smallskip 

Now, observe that: 
\begin{itemize}

\item In all cases, we have 
$
\id=\rees{F_0} = \contheta{{Q_0}} = \contheta{{Q_1}} \subset \rees{F_1} \subset \contheta{{Q_2}}  \subset \rees{F_2} 
\subset \cdots \subset  \contheta{{Q_{n-2}}}  \subset \rees{F_{n-2}};  
$

\item If $n\equiv \{0,3\}\mod 4$, then $\contheta{Q_k}=\conpi{\,Q_k}$, for $0\leqslant k\leqslant n$; 

\item If $n\equiv 0\mod 4$, then $\rees{F_{n-2}}=\contheta{Q_{n-1}} \subset \rees{F_{n-1}} \subset \contheta{Q_{n}}\subset \rees{F_{n}}=\omega$; 

\item If $n\equiv 3\mod 4$, then $\rees{F_{n-2}}=\contheta{Q_{n-1}} \subset \rees{F_{n-1}} = \contheta{Q_{n}}\subset \rees{F_{n}}=\omega$; 

\item If $n\equiv \{1,2\}\mod 4$, then $\contheta{Q_k}=\conpi{\,Q_k}$, for $0\leqslant k\leqslant n-2$, and $\contheta{Q_n}=\conpi{\,Q_n}$. 
In this case, we also have: 
$$
\rees{F_{n-2}} \subset \contheta{Q_{n-1}^\mathscr{o}} \subset \rees{F_{n-1}^\mathscr{o}} \subset \conpi{\,Q_{n-1}^\mathscr{e}} \subset \rees{F_{n-1}}, ~
\rees{F_{n-2}} \subset \contheta{Q_{n-1}^\mathscr{e}} \subset \rees{F_{n-1}^\mathscr{e}} \subset \conpi{\,Q_{n-1}^\mathscr{o}} \subset \rees{F_{n-1}}, ~
$$
$$
\conpi{\,Q_{n-1}^\mathscr{o}} \cap\, \conpi{\,Q_{n-1}^\mathscr{e}} = \contheta{Q_{n-1}^\mathscr{o}} \cup\, \contheta{Q_{n-1}^\mathscr{e}}, 
\rees{F_{n-1}^\mathscr{e}} \cap\, \conpi{\,Q_{n-1}^\mathscr{e}} =  \contheta{Q_{n-1}^\mathscr{e}}, ~ 
\rees{F_{n-1}^\mathscr{o}} \cap\, \conpi{\,Q_{n-1}^\mathscr{o}} =  \contheta{Q_{n-1}^\mathscr{o}}
$$
and 
$$
\rees{F_{n-1}^\mathscr{o}} \cap \rees{F_{n-1}^\mathscr{e}} = \rees{F_{n-2}} = \contheta{Q_{n-1}^\mathscr{o}}\cap \contheta{Q_{n-1}^\mathscr{e}}; 
$$

\item If $n\equiv 1\mod 4$, then $\rees{F_{n-1}} \subset \contheta{Q_{n}}\subset \rees{F_{n}}=\omega$; 

\item If $n\equiv 2\mod 4$, then $\rees{F_{n-1}} = \contheta{Q_{n}}\subset \rees{F_{n}}=\omega$. 
\end{itemize} 

Let $(P_1,\leqslant_1)$ and
$(P_2,\leqslant_2)$ be two disjoint posets. The ordinal sum of $P_1$
and $P_2$ (in this order) is the poset $P_1\oplus P_2$ with
base set $P_1\cup P_2$ and partial order $\leqslant$ defined by: for all
$x,y\in P_1\cup P_2$, we have $x\leqslant y$ if and only if $x\in P_1$ and $y\in
P_2$; or $x,y\in P_1$ and $x\leqslant_1 y$; or $x,y\in P_2$ and $x\leqslant_2
y$. Note that this operator on posets is associative but not commutative.

We have the following description of the congruences of $\AM_n$: 

\begin{theorem} \label{conAM} 
Let $C=\left\{  
\id \subset \rees{F_1} \subset \contheta{{Q_2}}  \subset \rees{F_2} 
\subset \cdots \subset  \contheta{{Q_{n-2}}}  \subset \rees{F_{n-2}}\right\}$. Then: 
\begin{enumerate}
\item If $n\equiv 0\mod 4$, then 
$\con(\AM_n)=C\oplus\left\{ \rees{F_{n-1}} \subset \contheta{Q_{n}}\subset \omega\right\}$; 

\item If $n\equiv 1\mod 4$, then 
$$
\con(\AM_n)=C\oplus\left\{
\contheta{Q_{n-1}^\mathscr{o}}, \contheta{Q_{n-1}^\mathscr{e}}, 
\contheta{Q_{n-1}^\mathscr{o}}\cup\, \contheta{Q_{n-1}^\mathscr{e}}, 
\rees{F_{n-1}^\mathscr{o}},  \rees{F_{n-1}^\mathscr{e}},
\conpi{\,Q_{n-1}^\mathscr{o}} , \conpi{\,Q_{n-1}^\mathscr{e}} \subset  
\rees{F_{n-1}} \subset \contheta{Q_{n}}\subset \omega
\right\};
$$ 

\item If $n\equiv 2\mod 4$, then 
$$
\con(\AM_n)=C\oplus\left\{
\contheta{Q_{n-1}^\mathscr{o}}, \contheta{Q_{n-1}^\mathscr{e}}, 
\contheta{Q_{n-1}^\mathscr{o}}\cup\, \contheta{Q_{n-1}^\mathscr{e}}, 
\rees{F_{n-1}^\mathscr{o}},  \rees{F_{n-1}^\mathscr{e}},
\conpi{\,Q_{n-1}^\mathscr{o}} , \conpi{\,Q_{n-1}^\mathscr{e}}  \subset 
\rees{F_{n-1}} \subset \omega
\right\};
$$ 

\item If $n\equiv 3\mod 4$, then $\con(\AM_n)=C\oplus\left\{ \rees{F_{n-1}} \subset \omega\right\}$.
\end{enumerate} 
\end{theorem} 
\begin{proof}
Let $\rho$ be a congruence of $\AM_n$. Let $\rho'=\rho\cap(\AO_n\times\AO_n)$. 
Then, by Theorem \ref{conAO}, we have $\rho'=\rees{I}$, for some ideal $I$ of $\AO_n$. 
Observe that, if $\alpha,\beta\in\AM_n$ are such that $\alpha\rho\beta$, then $\alpha^{-1}\rho\beta^{-1}$, 
whence $\alpha\alpha^{-1}\rho\beta\beta^{-1}$ and $\alpha^{-1}\alpha\rho\beta^{-1}\beta$, from which follows 
$\alpha\alpha^{-1}\rho'\beta\beta^{-1}$ and $\alpha^{-1}\alpha\rho'\beta^{-1}\beta$, 
since $\alpha\alpha^{-1},\alpha^{-1}\alpha,\beta\beta^{-1},\beta^{-1}\beta\in\AO_n$. 

\smallskip

If $I=\AO_n$, then $\rho'$ is the universal congruence of $\AO_n$ and so, in particular, $\id_n\rho'\emptyset$, whence $\id_n\rho\emptyset$. 
Therefore, in this case, we get 
$\alpha=\id_n\alpha\rho\emptyset\alpha=\emptyset$, for all $\alpha\in\AM_n$, i.e. $\rho$ is the universal congruence of $\AM_n$. 

\smallskip 

Next, we suppose that $I=I_k$, for some $0\leqslant k\leqslant n-1$. 

Let $\alpha\in F_k$. Then, $\alpha\alpha^{-1}\in I_k$ and so $\alpha\alpha^{-1}\rho'\emptyset$, 
whence  $\alpha\alpha^{-1}\rho\emptyset$, which implies $\alpha=\alpha\alpha^{-1}\alpha\rho\emptyset\alpha=\emptyset$. 
Thus, $\rees{F_k}\subseteq\rho$. 

Now, let $\alpha\in\AM_n\setminus F_k$ and take $\beta\in\AM_n$ such that $\alpha\rho\beta$. 
Then, $\alpha\alpha^{-1}\rho'\beta\beta^{-1}$ and $\alpha^{-1}\alpha\rho'\beta^{-1}\beta$, 
whence $\alpha\alpha^{-1}=\beta\beta^{-1}$ and $\alpha^{-1}\alpha=\beta^{-1}\beta$, 
since $\alpha\alpha^{-1},\alpha^{-1}\alpha\not\in I_k$. Therefore, $\alpha\mathscr{H}\beta$. 

At this point, observe that, if $k=n-1$, then $\rho=\contheta{Q_n}=\conpi{Q_n}$ or $\rho=\rees{F_{n-1}}$. 
Moreover, in this case, for $n\equiv\{2,3\}\mod 4$, we get exactly $\rho=\rees{F_{n-1}} (=\contheta{Q_n}=\conpi{Q_n})$, 
since $Q_n=\{\id_n\}$. 
Therefore, from now on, we suppose that $0\leqslant k\leqslant n-2$.

Let $\alpha,\beta\in\AM_n$ be such that $\alpha\rho\beta$ and $|\im(\alpha)|\geqslant k+2$. 
Then, as showed above, we have $\alpha\mathscr{H}\beta$ and so $\dom(\alpha)=\dom(\beta)$ and $\im(\alpha)=\im(\beta)$. 
Let $\varepsilon=\id_{\dom(\alpha)\setminus\{\min\dom(\alpha)\}}$. 
Then, $\varepsilon\in\AM_n$ and 
$\dom(\varepsilon\alpha)=\dom(\varepsilon\beta)=\dom(\varepsilon)=\dom(\alpha)\setminus\{\min\dom(\alpha)\}$, 
whence $|\im(\varepsilon\alpha)|\geqslant k+1$ and so $\varepsilon\alpha\not\in F_k$. 
Since $\alpha\rho\beta$, we also have $\varepsilon\alpha\rho\varepsilon\beta$ and so $\varepsilon\alpha\mathscr{H}\varepsilon\beta$. 
If $\alpha\neq\beta$, then $(\min\dom(\alpha))\alpha = \min\im(\alpha)$ and $(\min\dom(\alpha))\beta = \max\im(\alpha)$ or vice-versa, 
whence  $\im(\varepsilon\alpha)\neq\im(\varepsilon\beta)$, which is a contradiction since $\varepsilon\alpha\mathscr{H}\varepsilon\beta$. 
Therefore, $\alpha=\beta$. 

Suppose that $0\leqslant k\leqslant n-3$ and let $\alpha\in Q_{k+1}$ and $\beta\in\AM_n$ be such that $\alpha\mathscr{H}\beta$. 
Let 
$$
\alpha_L=\transf{1&\cdots&k+1\\a_1&\cdots&a_{k+1}}\quad\text{and}\quad \alpha_R=\transf{b_1&\cdots&b_{k+1}\\1&\cdots&k+1}, 
$$
where $\{a_1<\cdots<a_{k+1}\}=\dom(\alpha)=\dom(\beta)$ and $\{b_1<\cdots<b_{k+1}\}=\im(\alpha)=\im(\beta)$. 
Then, $\alpha_L,\alpha_R\in Q_{k+1}$ (notice that, $k+1\leqslant n-2$), 
$\alpha=\alpha_L^{-1}\alpha_L\alpha\alpha_R\alpha_R^{-1}$ and $\beta=\alpha_L^{-1}\alpha_L\beta\alpha_R\alpha_R^{-1}$. 
Let 
$$
\varepsilon_{k+1}=\transf{1&\cdots&k+1\\1&\cdots&k+1}=\alpha_L\alpha_L^{-1}=\alpha_R^{-1}\alpha_R
\quad\text{and}\quad \tau_{k+1}=\transf{1&\cdots&k+1\\k+1&\cdots&1}.  
$$
Then, $\varepsilon_{k+1},\tau_{k+1}\in Q_{k+1}$, $\varepsilon_{k+1}\mathscr{H}\tau_{k+1}$ and, if $\alpha\neq\beta$, then 
$$
\{\alpha_L\alpha\alpha_R,\alpha_L\beta\alpha_R\}=\{\varepsilon_{k+1},\tau_{k+1}\}
\quad\text{and}\quad 
\{\alpha_L^{-1}\varepsilon_{k+1}\alpha_R^{-1},\alpha_L^{-1}\tau_{k+1}\alpha_R^{-1}\}=\{\alpha,\beta\}, 
$$
whence 
$$
\alpha\rho\beta \Longrightarrow \alpha_L\alpha\alpha_R\rho\alpha_L\beta\alpha_R 
\Longleftrightarrow \varepsilon_{k+1}\rho\tau_{k+1} \Longrightarrow \alpha_L^{-1}\varepsilon_{k+1}\alpha_R^{-1}\rho\alpha_L^{-1}\tau_{k+1}\alpha_R^{-1}
\Longleftrightarrow \alpha\rho\beta, 
$$
i.e. $\alpha\rho\beta$ if and only if $\varepsilon_{k+1}\rho\tau_{k+1}$. 
Therefore, 
$$
(\varepsilon_{k+1},\tau_{k+1})\in\rho \Longrightarrow \rho=\contheta{Q_{k+1}}=\conpi{\,Q_{k+1}}
\quad\text{and}\quad  
(\varepsilon_{k+1},\tau_{k+1})\not\in\rho \Longrightarrow \rho=\rees{F_k}.
$$

Now, suppose that $k=n-2$. 
If $n\equiv \{0,3\}\mod 4$, then $Q_{n-1}$ has trivial $\mathscr{H}$-classes, 
and so $\alpha\rho\beta$ implies that $\alpha=\beta$, for all $\alpha\in Q_{n-1}$ and $\beta\in\AM_n$, 
whence $\rho=\rees{F_{n-2}}$. Therefore, suppose that $n\equiv \{1,2\}\mod 4$. 
Let $\alpha\in Q_{n-1}^\mathscr{o}$ [respectively, $\alpha\in Q_{n-1}^\mathscr{e}$]. 
Let 
$$
\mbox{$\alpha_L=\transf{2&\cdots&n\\a_1&\cdots&a_{n-1}}$ and $\alpha_R=\transf{b_1&\cdots&b_{n-1}\\2&\cdots&n}$  
[respectively, $\alpha_{L'}=\transf{1&3&\cdots&n\\a_1&a_2&\cdots&a_{n-1}}$ and $\alpha_{R'}=\transf{b_1&b_2&\cdots&b_{n-1}\\1&3&\cdots&n}$]},  
$$
where $\{a_1<\cdots<a_{n-1}\}=\dom(\alpha)$ and $\{b_1<\cdots<b_{n-1}\}=\im(\alpha)$. 
Since $(-1)^{\gd(\alpha)}=-1=(-1)^{\gi(\alpha)}$ and $\gd(\alpha_L)=1=\gi(\alpha_R)$ 
[respectively, $(-1)^{\gd(\alpha)}=1=(-1)^{\gi(\alpha)}$ and $\gd(\alpha_{L'})=2=\gi(\alpha_{R'})$], 
we have $\alpha_L,\alpha_R\in Q_{n-1}^\mathscr{o}$ [respectively, $\alpha_{L'},\alpha_{R'}\in Q_{n-1}^\mathscr{e}$]. 
Let 
$$
\mbox{
$\varepsilon_{n-1}=\transf{2&\cdots&n\\2&\cdots&n}=\alpha_L\alpha_L^{-1}=\alpha_R^{-1}\alpha_R$ and $\tau_{n-1}=\transf{2&\cdots&n\\n&\cdots&2}$ 
}
$$
$$
\mbox{
[respectively,   $\varepsilon'_{n-1}=\transf{1&3&\cdots&n\\1&3&\cdots&n}=\alpha_{L'}\alpha_{L'}^{-1}=\alpha_{R'}^{-1}\alpha_{R'}$ and 
$\tau'_{n-1}=\transf{1&3&\cdots&n&n-1&\\n&n-1&\cdots&3&1}$]. 
}
$$
Then, $\varepsilon_{n-1},\tau_{n-1}\in Q_{n-1}^\mathscr{o}$ and $\varepsilon_{n-1}\mathscr{H}\tau_{n-1}$ 
[respectively, $\varepsilon'_{n-1},\tau'_{n-1}\in Q_{n-1}^\mathscr{e}$ and $\varepsilon'_{n-1}\mathscr{H}\tau'_{n-1}$]. 
Let $\beta\in\AM_n$ be such that $\alpha\mathscr{H}\beta$. If $\alpha\neq\beta$, then 
$$
\{\alpha_L\alpha\alpha_R,\alpha_L\beta\alpha_R\}=\{\varepsilon_{n-1},\tau_{n-1}\}
\quad\text{and}\quad 
\{\alpha_L^{-1}\varepsilon_{n-1}\alpha_R^{-1},\alpha_L^{-1}\tau_{n-1}\alpha_R^{-1}\}=\{\alpha,\beta\} 
$$
and 
$$
\alpha\rho\beta \Longrightarrow \alpha_L\alpha\alpha_R\rho\alpha_L\beta\alpha_R 
\Longleftrightarrow \varepsilon_{n-1}\rho\tau_{n-1} \Longrightarrow \alpha_L^{-1}\varepsilon_{n-1}\alpha_R^{-1}\rho\alpha_L^{-1}\tau_{n-1}\alpha_R^{-1}
\Longleftrightarrow \alpha\rho\beta, 
$$
i.e. $\alpha\rho\beta$ if and only if $\varepsilon_{n-1}\rho\tau_{n-1}$
[respectively, in a similar way, we deduce that $\alpha\rho\beta$ if and only if $\varepsilon'_{n-1}\rho\tau'_{n-1}$]. 
Therefore: 
\begin{description}
\item If $(\varepsilon_{n-1},\tau_{n-1})\not\in\rho$ and $(\varepsilon'_{n-1},\tau'_{n-1})\not\in\rho$, then $\rho=\rees{F_{n-2}}$; 
\item If $(\varepsilon_{n-1},\tau_{n-1})\in\rho$ and $(\varepsilon'_{n-1},\tau'_{n-1})\not\in\rho$, then $\rho=\contheta{Q_{n-1}^\mathscr{o}}$; 
\item If $(\varepsilon_{n-1},\tau_{n-1})\not\in\rho$ and $(\varepsilon'_{n-1},\tau'_{n-1})\in\rho$, then $\rho=\contheta{Q_{n-1}^\mathscr{e}}$; 
\item If $(\varepsilon_{n-1},\tau_{n-1})\in\rho$ and $(\varepsilon'_{n-1},\tau'_{n-1})\in\rho$, then $\rho=\contheta{Q_{n-1}^\mathscr{o}}\cup\contheta{Q_{n-1}^\mathscr{e}}$ 
(observe that, $A(Q_{n-1}^\mathscr{o})=F_{n-2}=A(Q_{n-1}^\mathscr{e})$). 
\end{description}
 
\smallskip 

At this stage, it remains to consider the cases $I=I_{n-1}^\mathscr{o}$ and $I=I_{n-1}^\mathscr{e}$. 
So, let us suppose that $I=I_{n-1}^\mathscr{o}$ [respectively, $I=I_{n-1}^\mathscr{e}$]. 

Let $\alpha\in F_{n-2}$. Then, $\alpha\alpha^{-1}\in I_{n-2}$ and so  $\alpha\alpha^{-1}\rho'\emptyset$, whence  $\alpha\alpha^{-1}\rho\emptyset$, 
from which follows $\alpha=\alpha\alpha^{-1}\alpha\rho\emptyset\alpha=\emptyset$. 

Let $\alpha\in Q_n$ and let $\beta\in\AM_n$ be such that $\alpha\rho\beta$. 
Then, $\alpha\alpha^{-1}\rho'\beta\beta^{-1}$ and so $\beta\beta^{-1}=\alpha\alpha^{-1}=\id_n$, whence $\beta\in Q_n$. 
Notice that, if $n\equiv\{2,3\}\mod4$, then $Q_n=\{\id_n\}$ and so $\alpha=\beta$. 

Suppose that $n\equiv \{1,2\}\mod 4$. 

If $\alpha\in Q_{n-1}^\mathscr{o}$ [respectively,  $\alpha\in Q_{n-1}^\mathscr{e}$], then $\alpha\alpha^{-1}\in I_{n-1}^\mathscr{o}$ 
[respectively, $\alpha\alpha^{-1}\in I_{n-1}^\mathscr{e}$], whence $\alpha\alpha^{-1}\rho'\emptyset$ and so $\alpha\alpha^{-1}\rho\emptyset$, 
from which follows $\alpha=\alpha\alpha^{-1}\alpha\rho\emptyset\alpha=\emptyset$. 
Therefore, $\rees{F_{n-1}^\mathscr{o}}\subseteq\rho$ [respectively, $\rees{F_{n-1}^\mathscr{e}}\subseteq\rho$]. 

Next, let $\alpha\in Q_{n-1}^\mathscr{e}$ [respectively,  $\alpha\in Q_{n-1}^\mathscr{o}$] and $\beta\in\AM_n$ be such that $\alpha\rho\beta$. 
Then, 
$\alpha\alpha^{-1}\rho'\beta\beta^{-1}$ and $\alpha^{-1}\alpha\rho'\beta^{-1}\beta$, 
whence $\beta\beta^{-1}=\alpha\alpha^{-1}$ and $\beta^{-1}\beta=\alpha^{-1}\alpha$, 
since $\alpha\alpha^{-1},\alpha^{-1}\alpha\in I_{n-1}^\mathscr{e}$ [respectively, $\alpha\alpha^{-1},\alpha^{-1}\alpha\in I_{n-1}^\mathscr{o}$]. 
Thus, $\alpha\mathscr{H}\beta$. 
Conversely, let $\alpha,\beta\in Q_{n-1}^\mathscr{e}$ [respectively,  $\alpha,\beta\in Q_{n-1}^\mathscr{o}$] be such that $\alpha\mathscr{H}\beta$.
We can then deduce, in a similar way to the case of $I=I_{n-2}$, that 
$$
\mbox{
$\alpha\rho\beta$ if and only if $\varepsilon'_{n-1}\rho\tau'_{n-1}$
[respectively, $\alpha\rho\beta$ if and only if $\varepsilon_{n-1}\rho\tau_{n-1}$]}, 
$$ 
where $\varepsilon'_{n-1}, \tau'_{n-1}\in Q_{n-1}^\mathscr{e}$ [respectively, $\varepsilon_{n-1}, \tau_{n-1}\in Q_{n-1}^\mathscr{o}$] 
are as in the case $I=I_{n-2}$. 

Now, for $n\equiv 1\mod4$, if $h\rho\id_n$, then
$$
\mbox{
$\transf{1&3&\cdots&n\\n&n-2&\cdots&1}=\varepsilon'_{n-1}h\rho \varepsilon'_{n-1}\id_n=\varepsilon'_{n-1}$ 
and $(\varepsilon'_{n-1}h,\varepsilon'_{n-1})\not\in\mathscr{H}$
} 
$$
$$
\mbox{[respectively, 
$\transf{2&\cdots&n\\n-1&\cdots&1}=\varepsilon_{n-1}h\rho \varepsilon_{n-1}\id_n=\varepsilon_{n-1}$ and $(\varepsilon_{n-1}h,\varepsilon_{n-1})\not\in\mathscr{H}$], 
} 
$$
which is a contradiction, whence $(h,\id_n)\not\in\rho$. 

Therefore, if $\varepsilon'_{n-1}\rho\tau'_{n-1}$ [respectively, $\varepsilon_{n-1}\rho\tau_{n-1}$], then $\rho=\conpi{\,Q_{n-1}^\mathscr{e}}$ 
[respectively, $\rho=\conpi{\,Q_{n-1}^\mathscr{o}}$], otherwise $\rho=\rees{F_{n-1}^\mathscr{o}}$ [respectively, $\rho=\rees{F_{n-1}^\mathscr{e}}$]. 

Finally, suppose that $n\equiv \{0,3\}\mod 4$. 

If $\alpha\in J_{n-1}^\mathscr{o}$ [respectively,  $\alpha\in J_{n-1}^\mathscr{e}$], then $\alpha\rho'\emptyset$ and so $\alpha\rho\emptyset$. 
Next, let $\alpha\in J_{n-1}^\mathscr{e}$ [respectively,  $\alpha\in J_{n-1}^\mathscr{o}$] and suppose that 
$\dom(\alpha)=\{a_1<\cdots<a_{n-1}\}$ and $\im(\alpha)=\{b_1<\cdots<b_{n-1}\}$. 
Take 
$$
\mbox{
$\alpha_L=\transf{2&\cdots&n\\a_{n-1}&\cdots&a_1}$ and $\alpha_R=\transf{b_1&\cdots&b_{n-1}\\n&\cdots&2}$ 
}
$$
$$
\mbox{
[respectively, $\alpha_L=\transf{1&3&4&\cdots&n\\a_{n-1}&a_{n-2}&a_{n-3}&\cdots&a_1}$ and $\alpha_R=\transf{b_1&\cdots&b_{n-3}&b_{n-2}&b_{n-1}\\n&\cdots&4&3&1}$].  
}
$$
Since $(-1)^{\gd(\alpha)}=(-1)^{\gi(\alpha)}=1$ and $\gd(\alpha_L)=\gi(\alpha_R)=1$ 
[respectively, $(-1)^{\gd(\alpha)}=(-1)^{\gi(\alpha)}=-1$ and $\gd(\alpha_L)=\gi(\alpha_R)=2$], we have $\alpha_L,\alpha_R\in Q_{n-1}$. 
Moreover, 
$$
\mbox{
$\alpha_L\alpha\alpha_R=\transf{2&\cdots&n\\2&\cdots&n}\in J_{n-1}^\mathscr{o}$ 
[respectively, $\alpha_L\alpha\alpha_R=\transf{1&3&4&\cdots&n\\1&3&4&\cdots&n}\in J_{n-1}^\mathscr{e}$],  
}
$$
whence $\alpha_L\alpha\alpha_R\rho'\emptyset$ and so $\alpha_L\alpha\alpha_R\rho\emptyset$, 
from which follows $\alpha = \alpha_L^{-1}\alpha_L\alpha\alpha_R\alpha_R^{-1}\rho  \alpha_L^{-1}\emptyset\alpha_R^{-1}=\emptyset$. 
Now, if $\alpha\in Q_{n-1}$, then $\alpha\alpha^{-1}\in J_{n-1}^\mathscr{o}\cup J_{n-1}^\mathscr{e}$, whence $\alpha\alpha^{-1}\rho\emptyset$ and so 
$\alpha=\alpha\alpha^{-1}\alpha\rho\emptyset\alpha=\emptyset$. 
Therefore, $\rees{F_{n-1}}\subseteq\rho$ and thus $\rho=\rees{F_{n-1}}$ or $\rho=\contheta{Q_n}=\conpi{\,Q_n}$. 
In particular, for $n\equiv 3\mod 4$, we get exactly $\rho=\rees{F_{n-1}} (=\contheta{Q_n}=\conpi{Q_n})$, 
since $Q_n=\{\id_n\}$. 

The proof is now complete. 
\end{proof} 

Observe that, for $n\equiv \{1,2\}\mod 4$, the interval $[\rees{F_{n-2}},\rees{F_{n-1}}]$ of the lattice $\con(\AM_n)$ can be represented by the Hasse diagram of Figure \ref{intcon}. 
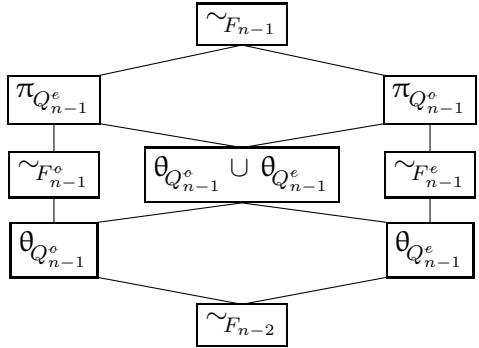
\begin{figure}[H] 
\centering
\begin{tikzpicture}[scale=0.5]
\draw (0,0) node{\fbox{$\rees{F_{n-1}}$}} ;  
\draw (-5,-2) node{\fbox{$\conpi{\,Q_{n-1}^\mathscr{e}}$}} ; 
\draw (5,-2) node{\fbox{$\conpi{\,Q_{n-1}^\mathscr{o}}$}} ; 
\draw (-5,-4) node{\fbox{$\rees{F_{n-1}^\mathscr{o}}$}} ; 
\draw (5,-4) node{\fbox{$\rees{F_{n-1}^\mathscr{e}}$}} ; 
\draw (-5,-6) node{\fbox{$\contheta{Q_{n-1}^\mathscr{o}}$}} ; 
\draw (5,-6) node{\fbox{$\contheta{Q_{n-1}^\mathscr{e}}$}} ; 
\draw (0,-8) node{\fbox{$\rees{F_{n-2}}$}} ; 
\draw (0,-4) node{\fbox{$\contheta{Q_{n-1}^\mathscr{o}}\cup\,\contheta{Q_{n-1}^\mathscr{e}}$}} ; 
\draw[thin] (0,-.55) -- (-3.8,-1.36);
\draw[thin] (0,-.55) -- (3.8,-1.36);
\draw[thin] (0,-3.28) -- (-3.8,-2.64);
\draw[thin] (0,-3.28) -- (3.8,-2.64);
\draw[thin] (0,-4.75) -- (-3.85,-5.27);
\draw[thin] (0,-4.75) -- (3.85,-5.27);
\draw[thin] (0,-7.45) -- (-3.85,-6.73);
\draw[thin] (0,-7.45) -- (3.85,-6.73);
\draw[thin] (-5,-2.65) -- (-5,-3.4);
\draw[thin] (5,-2.65) -- (5,-3.4);
\draw[thin] (-5,-4.6) -- (-5,-5.27);
\draw[thin] (5,-4.6) -- (5,-5.27);
\end{tikzpicture}
\caption{The Hasse diagram of $[\rees{F_{n-2}},\rees{F_{n-1}}]$, for $n\equiv \{1,2\}\mod 4$.} \label{intcon}
\end{figure}

We finish this section by noticing that 
$$
|\con(\AM_n)|=
\left\{
\begin{array}{ll}
2n-1 & \mbox{if $n\equiv 0\mod 4$}\\
2n+6 & \mbox{if $n\equiv 1\mod 4$}\\
2n+5 & \mbox{if $n\equiv 2\mod 4$}\\
2n-2 & \mbox{if $n\equiv 3\mod 4$}. 
\end{array}
\right.
$$

\section{Generators and rank of $\AO_n$} \label{rAO} 

Let $X_i=\Omega_n\setminus\{i\}$, for $1\leqslant i\leqslant n$. Let $x_1,x_2,\ldots,x_n\in\POI_n$ defined by  
$$
x_1=\left\{\begin{array}{ll}
\transf{2&\cdots&n\\1&\cdots&n-1}=\transf{X_1\\X_n} & \mbox{if $n$ is odd}\\
\transf{2&\cdots&n-1&n\\1&\cdots&n-2&n}=\transf{X_1\\X_{n-1}} & \mbox{if $n$ is even}
\end{array}\right.,
\quad 
x_2=\left\{\begin{array}{ll}
\transf{1&3&\cdots&n-1&n\\1&2&\cdots&n-2&n}=\transf{X_2\\X_{n-1}} & \mbox{if $n$ is odd}\\
\transf{1&3&\cdots&n\\1&2&\cdots&n-1}=\transf{X_2\\X_n} & \mbox{if $n$ is even}
\end{array}\right.
$$
and
$$
x_i=\transf{1&\cdots&i-3&i-2&i-1&i+1&\cdots&n\\1&\cdots&i-3&i-1&i&i+1&\cdots&n}=\transf{X_i\\X_{i-2}}, ~\mbox{for $3\leqslant i\leqslant n$}. 
$$
By Proposition \ref{chAO}, it is clear that $x_1,x_2,\ldots,x_n\in\AO_n$. Our next objective is to show that $x_1,x_2,\ldots,x_n$ generate $\AO_n$. 
We begin by presenting a series of lemmas.

First, we recall that, in the proof of \cite[Lemma 2.7]{Fernandes:2001}, we proved the following result.

\begin{lemma}\label{n-2}
Let $0\leqslant k\leqslant n-3$. Then, every element of $J_k$ is a product of elements of $J_{k+1}$. 
\end{lemma}

As a consequence of the last lemma, we conclude that every element of $J_k$ is a product of elements of $J_{n-2}$, for  $0\leqslant k\leqslant n-2$. 

\begin{lemma}\label{n-1}
$J_{n-1}^\mathscr{o}\cup J_{n-1}^\mathscr{e}\subseteq\langle x_1,x_2,\ldots,x_n\rangle$. 
\end{lemma}
\begin{proof}
Take $\alpha\in J_{n-1}^\mathscr{o}\cup J_{n-1}^\mathscr{e}$. 
Then, $\alpha=\transf{X_i\\X_j}$, for some $1\leqslant i,j\leqslant n$ with the same parity. 

If $j<i$, then $j+2\leqslant i$ and $\alpha=x_ix_{i-2}\cdots x_{j+2}$. 

Next, suppose that $i\leqslant j$ and $i,j$ are odd. 
If $j\geqslant n-1$, then $\alpha=x_ix_{i-2}\cdots x_1$.
If $j\leqslant n-2$, then $\alpha=x_ix_{i-2}\cdots x_1x_nx_{n-2}\cdots x_{j+2}$, if $n$ is odd, 
and  $\alpha=x_ix_{i-2}\cdots x_1x_{n-1}x_{n-3}\cdots x_{j+2}$, if $n$ is even. 

Finally, suppose that $i\leqslant j$ and $i,j$ are even. 
If $j\geqslant n-1$, then $\alpha=x_ix_{i-2}\cdots x_2$.
If $j\leqslant n-2$, then $\alpha=x_ix_{i-2}\cdots x_2x_{n-1}x_{n-3}\cdots x_{j+2}$, if $n$ is odd, 
and  $\alpha=x_ix_{i-2}\cdots x_2x_nx_{n-2}\cdots x_{j+2}$, if $n$ is even. 
\end{proof} 

\begin{lemma}\label{n2c}
Let $\alpha\in J_{n-2}$ be such that $\dom(\alpha)=\Omega_n\setminus\{1,2i\}$ and $\im(\alpha)=\Omega_n\setminus\{1,2j+1\}$,  
for some $1\leqslant i\leqslant \lfloor\frac{n}{2}\rfloor$ and $1\leqslant j\leqslant \lfloor\frac{n-1}{2}\rfloor$. 
Then, there exist $\beta\in J_{n-1}^\mathscr{e}$ and $\gamma\in J_{n-1}^\mathscr{o}$ such that $\alpha=\beta\gamma$. 
\end{lemma}
\begin{proof}
Let 
$
\beta=\transf{1&2&\cdots&2i-1&2i+1&\cdots&n\\1&3&\cdots&2i&2i+1&\cdots&n}\in J_{n-1}^\mathscr{e}
$
and 
$
\gamma=\transf{2&\cdots&2j+1&2j+2&\cdots&n\\1&\cdots&2j&2j+2&\cdots&n} \in J_{n-1}^\mathscr{o}. 
$
Then, $\dom(\beta\gamma)=\dom(\alpha)$ and $\im(\beta\gamma)=\im(\alpha)$, whence $\alpha=\beta\gamma$,
as required. 
\end{proof}

\begin{lemma}\label{n2d2}
Let $\alpha\in J_{n-2}$ be such that $1\not\in\dom(\alpha)\cup\im(\alpha)$.  
Then, there exists $\beta,\gamma\in J_{n-1}^\mathscr{o}\cup  J_{n-1}^\mathscr{e}$ such that $\alpha=\beta\gamma$. 
\end{lemma}
\begin{proof}
Let $2\leqslant i,j\leqslant n$ be such that $\dom(\alpha)=\Omega_n\setminus\{1,i\}$ and $\im(\alpha)=\Omega_n\setminus\{1,j\}$. 

First, suppose that $i$ and $j$ have the same parity. 
Let $\beta\in\POI_n$ be such that $\dom(\beta)=\Omega_n\setminus\{i\}$ and $\im(\beta)=\Omega_n\setminus\{j\}$ 
and let $\gamma=\id_{X_1}$. Then,  $\beta,\gamma\in J_{n-1}^\mathscr{o}\cup  J_{n-1}^\mathscr{e}$ and $\alpha=\beta\gamma$. 

Secondly, suppose that $i$ and $j$ have distinct parities. If $i$ is even and $j$ is odd, by Lemma \ref{n2c}, 
there exist $\beta\in J_{n-1}^\mathscr{e}$ and $\gamma\in J_{n-1}^\mathscr{o}$ such that $\alpha=\beta\gamma$. 
On the other hand, if $j$ is even and $i$ is odd, by Lemma \ref{n2c}, 
there exist $\beta\in J_{n-1}^\mathscr{e}$ and $\gamma\in J_{n-1}^\mathscr{o}$ such that $\alpha^{-1}=\beta\gamma$, 
whence $\alpha=\gamma^{-1}\beta^{-1}$, with $\beta^{-1}\in J_{n-1}^\mathscr{e}$ and $\gamma^{-1}\in J_{n-1}^\mathscr{o}$, 
as required.
\end{proof}

\begin{lemma}\label{n2a}
Let $\alpha\in J_{n-2}$ be such that  
both elements of $\Omega_n\setminus\im(\alpha)$ are even numbers. 
Then, there exists $\beta\in J_{n-1}^\mathscr{e}$ such that $\alpha\beta\in J_{n-2}$ 
and $\Omega_n\setminus\im(\alpha\beta)$ contains an odd number. 
\end{lemma}
\begin{proof}
Let $1\leqslant i<j\leqslant \lfloor\frac{n}{2}\rfloor$ be such that $\im(\alpha)=\Omega_n\setminus\{2i,2j\}$. 
Let 
$
\beta=\transf{1&\cdots&2i-1&2i &2i+1&\cdots&2j-1&2j+1&\cdots&n\\1&\cdots&2i-1&2i+1&2i+2&\cdots&2j&2j+1&\cdots&n}. 
$
Then, $\beta\in J_{n-1}^\mathscr{e}$ and $\im(\alpha\beta)=\Omega_n\setminus\{2i,2i+1\}$, as required. 
\end{proof}

\begin{lemma}\label{n2b}
Let $\alpha\in J_{n-2}$ be such that 
$\Omega_n\setminus\im(\alpha)$ contains an odd number. 
Then, there exists $\beta\in J_{n-1}^\mathscr{o}$ such that $\alpha\beta\in J_{n-2}$ and $1\not\in\im(\alpha\beta)$. 
\end{lemma}
\begin{proof}
By hypothesis, there exists $0\leqslant i\leqslant \lfloor\frac{n-1}{2}\rfloor$ such that $2i+1\not\in\im(\alpha)$. 
Let 
$
\beta=\transf{1&\cdots&2i&2i+2&n\\2&\cdots&2i+1&2i+2&n}. 
$
Then, $\beta\in J_{n-1}^\mathscr{o}$ and, as $\im(\alpha)\subseteq\dom(\beta)$ and $1\not\in\im(\beta)$, 
we have $\alpha\beta\in J_{n-2}$ and $1\not\in\im(\alpha\beta)$, as required. 
\end{proof}

\begin{proposition}\label{genAO}
The set $\{x_1,x_2,\ldots,x_n\}$ generates the monoid $\AO_n$. 
\end{proposition}
\begin{proof}
In view of Lemmas \ref{n-2} and \ref{n-1}, it suffices to show that $J_{n-2}\subseteq \langle J_{n-1}^\mathscr{o}\cup J_{n-1}^\mathscr{e}\rangle$. 
Let $\alpha\in J_{n-2}$. 

If both elements of $\Omega_n\setminus\im(\alpha)$ are even numbers, by Lemma \ref{n2a}, 
there is $\beta_1\in J_{n-1}^\mathscr{e}$ such that $\alpha\beta_1\in J_{n-2}$ 
and $\Omega_n\setminus\im(\alpha\beta_1)$ contains an odd number. 
On the other hand, if $\Omega_n\setminus\im(\alpha)$ contains an odd number, then fix any $x\in \Omega_n\setminus\im(\alpha)$ and 
take $\beta_1=\id_{\im(\alpha)\cup\{x\}}\in J_{n-1}^\mathscr{o}\cup J_{n-1}^\mathscr{e}$. 
In this last case, $\alpha=\alpha\beta_1$ and so we also have that $\alpha\beta_1\in J_{n-2}$ 
and $\Omega_n\setminus\im(\alpha\beta_1)$ contains an odd number. 
Therefore, in either case, by Lemma \ref{n2b}, there exists $\beta_2\in J_{n-1}^\mathscr{o}$ 
such that $\alpha\beta_1\beta_2\in J_{n-2}$ and $1\not\in\im(\alpha\beta_1\beta_2)$. 

Suppose that $1\not\in\dom(\alpha\beta_1\beta_2)$. Then, by Lemma \ref{n2d2}, 
we have $\alpha\beta_1\beta_2\in\langle J_{n-1}^\mathscr{o}\cup J_{n-1}^\mathscr{e}\rangle$. 
Since $\alpha,\alpha\beta_1\beta_2\in J_{n-2}$, 
then $\alpha=(\alpha\beta_1\beta_2)\beta_2^{-1}\beta_1^{-1}$, with $\beta_1^{-1},\beta_2^{-1}\in J_{n-1}^\mathscr{o}\cup J_{n-1}^\mathscr{e}$, 
and so $\alpha\in\langle J_{n-1}^\mathscr{o}\cup J_{n-1}^\mathscr{e}\rangle$. 

On the other hand, suppose that $1\in\dom(\alpha\beta_1\beta_2)$ and let $\beta=(\alpha\beta_1\beta_2)^{-1}$. Since $\beta\in J_{n-2}$, we can apply Lemmas 
\ref{n2a} and \ref{n2b} to $\beta$ in a similar way to what we did for $\alpha$ above and thus take $\gamma_1,\gamma_2 \in J_{n-1}^\mathscr{o}\cup J_{n-1}^\mathscr{e}$ 
such that $\beta\gamma_1\gamma_2\in J_{n-2}$ and $1\not\in\im(\beta\gamma_1\gamma_2)$. 

If $1\in\dom(\beta\gamma_1\gamma_2)$, then $1\in\dom(\beta)$, whence $1\in\im(\alpha\beta_1\beta_2)$, which is a contradiction. 
So, $1\not\in\dom(\beta\gamma_1\gamma_2)$ and then, by Lemma \ref{n2d2}, 
we get $\beta\gamma_1\gamma_2\in\langle J_{n-1}^\mathscr{o}\cup J_{n-1}^\mathscr{e}\rangle$. 
Since $\beta,\beta\gamma_1\gamma_2\in J_{n-2}$, 
then $\beta=(\beta\gamma_1\gamma_2)\gamma_2^{-1}\gamma_1^{-1}$, 
with $\gamma_1^{-1},\gamma_2^{-1}\in J_{n-1}^\mathscr{o}\cup J_{n-1}^\mathscr{e}$. 
Therefore, $\beta\in\langle J_{n-1}^\mathscr{o}\cup J_{n-1}^\mathscr{e}\rangle$, 
whence $\alpha\beta_1\beta_2=\beta^{-1}\in  \langle J_{n-1}^\mathscr{o}\cup J_{n-1}^\mathscr{e}\rangle$ and so, as above, 
we may deduce that $\alpha=(\alpha\beta_1\beta_2)\beta_2^{-1}\beta_1^{-1}\in \langle J_{n-1}^\mathscr{o}\cup J_{n-1}^\mathscr{e}\rangle$, 
as required. 
\end{proof}

Since the group of units of $\AO_n$ is trivial, every generating set of $\AO_n$ must have at least one element whose domain is $X_i$, for all $1\leqslant i\leqslant n$, 
and so must contain at least $n$ elements. Thus, in view of Proposition \ref{genAO}, we have immediately the following result. 

\begin{theorem}\label{rkAO}
The monoid $\AO_n$ has rank $n$. 
\end{theorem}

\section{Generators and rank of $\AM_n$}\label{rAM}

We begin this section by observing that, as $\AO_n\subseteq\AM_n$, we also have $x_1,\ldots,x_n\in\AM_n$. Furthermore,
$$
hx_ih=x_{n-i+3}^{-1},~\mbox{for $3\leqslant i\leqslant n$,} \quad 
hx_1h=\left\{
\begin{array}{ll}
x_1^{-1} & \mbox{if $n$ is odd}\\
x_2^{-1} & \mbox{if $n$ is even}
\end{array}\right.
\quad\text{and}\quad
hx_2h=\left\{
\begin{array}{ll}
x_2^{-1} & \mbox{if $n$ is odd}\\
x_1^{-1} & \mbox{if $n$ is even.}
\end{array}\right.
$$

Suppose that $n\equiv\{0,1\}\mod 4$. Hence, $h\in\AM_n$. 
Let $\alpha\in\AM_n\setminus\AO_n$. Then, $h\alpha\in\AI_n\cap\POI_n$, 
i.e. $h\alpha\in\AO_n$ and so $\alpha=h(h\alpha)\in\langle \AO_n,h\rangle$. 
Thus, we have:

\begin{lemma}\label{hgen}
If  $n\equiv\{0,1\}\mod 4$, then $\AM_n=\langle \AO_n,h\rangle$. 
\end{lemma}

On the other hand, we also have: 

\begin{lemma}\label{hrk01}
If  $n\equiv\{0,1\}\mod 4$, then $\rank(\AM_n)\geqslant \lceil\frac{n}{2}\rceil+1$. 
\end{lemma}
\begin{proof}
Let $\mathscr{X}$ be a set of generators of $\AM_n$. Then, clearly, we must have $h\in\mathscr{X}$. 
On the other hand, let $1\leqslant i\leqslant n$. Since $x_i\in\AM_n$ and $x_i\not\in Q_n$, 
there exist $k\geqslant1$ and $\alpha_1,\ldots,\alpha_k\in\mathscr{X}$ such that $\alpha_1\neq h$ and 
$x_i=\alpha_1\cdots\alpha_k$ or $x_i=h\alpha_1\cdots\alpha_k$. 
Therefore, $\dom(\alpha_1)=\dom(x_i)=X_i$ or $\dom(\alpha_1)=\dom(hx_i)=X_{n-i+1}$. 
It follows that $\mathscr{X}$ also possesses at least $\lceil\frac{n}{2}\rceil$ elements of rank $n-1$ with distinct domains and so  
$\rank(\AM_n)\geqslant \lceil\frac{n}{2}\rceil+1$. 
\end{proof}

Now, suppose that $n\equiv\{2,3\}\mod 4$. Recall that, in this case, we have $h\not\in\AM_n$. 
Therefore, since the group of units of $\AM_n$ is trivial, just like for $\AO_n$, 
we can say right away that: 

\begin{lemma}\label{hrk23}
If  $n\equiv\{2,3\}\mod 4$, then $\rank(\AM_n)\geqslant n$. 
\end{lemma}

\smallskip 

Consider $n\equiv0\mod 4$. Then, observe that, both $n$ and $\frac{n}{2}$ are even integers. 
In this case, we have the following result. 

\begin{theorem}\label{genAM0}
If $n\equiv0\mod 4$, then $\{h, x_n,x_{n-1},\ldots,x_{\frac{n}{2}+4}, x_{\frac{n}{2}+3}, 
x_{\frac{n}{2}+2}x_{\frac{n}{2}}\cdots x_6x_4, 
x_{\frac{n}{2}+1}x_{\frac{n}{2}-1}\cdots x_5x_3 \}$ is a generating set of $\AM_n$ with minimum size.
In particular,  $\rank(\AM_n)=\frac{n}{2}+1$. 
\end{theorem} 
\begin{proof}
Let $M=\langle h, x_n,x_{n-1},\ldots,x_{\frac{n}{2}+4}, x_{\frac{n}{2}+3}, 
x_{\frac{n}{2}+2}x_{\frac{n}{2}}\cdots x_6x_4, x_{\frac{n}{2}+1}x_{\frac{n}{2}-1}\cdots x_5x_3 \rangle$. 
In view of Proposition \ref{genAO} and Lemmas \ref{hgen} and \ref{hrk01}, to prove this theorem, it suffices to show that $x_1,\ldots,x_{\frac{n}{2}+2}\in M$. 

First, observe that $x_{n-i+3}^{-1}=hx_ih\in M$, for $i=\frac{n}{2}+3,\ldots,n$, i.e. $x_i^{-1}\in M$, for $3\leqslant i\leqslant \frac{n}{2}$. 
Hence,
$$
x_{\frac{n}{2}+1} = (x_{\frac{n}{2}+1}x_{\frac{n}{2}-1}\cdots x_5x_3)x_3^{-1}x_5^{-1}\cdots x_{\frac{n}{2}-1}^{-1}\in M
\quad\text{and}\quad 
x_{\frac{n}{2}+2} = (x_{\frac{n}{2}+2}x_{\frac{n}{2}}\cdots x_6x_4)x_4^{-1}x_6^{-1}\cdots x_{\frac{n}{2}}^{-1} \in M. 
$$
It follows that 
$$
x_{\frac{n}{2}+1}^{-1} =h x_{\frac{n}{2}+2} h\in M
\quad\text{and}\quad 
x_{\frac{n}{2}+2}^{-1} = h x_{\frac{n}{2}+1} h \in M. 
$$
Thus,  
$$
x_{2i+1} = x_{2i+3}^{-1}x_{2i+5}^{-1}\cdots x_{\frac{n}{2}+1}^{-1} (x_{\frac{n}{2}+1}x_{\frac{n}{2}-1}\cdots x_5x_3)x_3^{-1}x_5^{-1}\cdots x_{2i-1}^{-1}\in M
$$
and 
$$
x_{2i+2} = x_{2i+4}^{-1}x_{2i+6}^{-1}\cdots x_{\frac{n}{2}+2}^{-1} (x_{\frac{n}{2}+2}x_{\frac{n}{2}}\cdots x_6x_4)x_4^{-1}x_6^{-1}\cdots x_{2i}^{-1} \in M 
$$
for $2\leqslant i\leqslant \frac{n}{4}-1$, 
i.e. $x_i\in M$, for $5\leqslant i\leqslant\frac{n}{2}$.  
In addition, 
$$
x_3 = x_{5}^{-1}x_{7}^{-1}\cdots x_{\frac{n}{2}+1}^{-1} (x_{\frac{n}{2}+1}x_{\frac{n}{2}-1}\cdots x_5x_3)\in M 
\quad\text{and}\quad 
x_{4} = x_{6}^{-1}x_{8}^{-1}\cdots x_{\frac{n}{2}+2}^{-1} (x_{\frac{n}{2}+2}x_{\frac{n}{2}}\cdots x_6x_4) \in M, 
$$ 
Finally,  
$$
x_1^{-1}=x_{n-1}x_{n-3}\cdots x_5x_3 \in M\quad\text{and}\quad x_2^{-1}=x_nx_{n-2}\cdots x_6x_4\in M, 
$$
whence $x_1=hx_2^{-1}h\in M$ and $x_2=hx_1^{-1}h\in M$, as required.
\end{proof}

\smallskip 

Now, we consider that $n\equiv1\mod 4$. In this case, both $n$ and $\frac{n+1}{2}$ are odd integers. 

\begin{theorem}\label{genAM1}
If $n\equiv1\mod 4$, then $\{h, x_n,x_{n-1},\ldots,x_{\frac{n+7}{2}}, x_{\frac{n+5}{2}}, 
x_{\frac{n+3}{2}}x_{\frac{n-1}{2}}\cdots x_6x_4, 
x_{\frac{n+1}{2}}x_{\frac{n-3}{2}}\cdots x_5x_3 \}$ is a generating set of $\AM_n$ with minimum size.
In particular,  $\rank(\AM_n)=\frac{n+1}{2}+1$. 
\end{theorem} 
\begin{proof}
The proof of this theorem will be quite analogous to the previous one. 

Let $M=\langle h, x_n,x_{n-1},\ldots,x_{\frac{n+7}{2}}, x_{\frac{n+5}{2}}, 
x_{\frac{n+3}{2}}x_{\frac{n-1}{2}}\cdots x_6x_4, x_{\frac{n+1}{2}}x_{\frac{n-3}{2}}\cdots x_5x_3 \rangle$. 
Bearing in mind Proposition \ref{genAO} and Lemmas \ref{hgen} and \ref{hrk01},  we just have to show that $x_1,\ldots,x_{\frac{n+3}{2}}\in M$. 

Since $x_{n-i+3}^{-1}=hx_ih\in M$, for $i=\frac{n+5}{2},\ldots,n$, i.e. $x_i^{-1}\in M$, for $3\leqslant i\leqslant \frac{n+1}{2}$, 
we have 
$$
x_{\frac{n+3}{2}} = (x_{\frac{n+3}{2}}x_{\frac{n-1}{2}}\cdots x_6x_4)x_4^{-1}x_6^{-1}\cdots x_{\frac{n-1}{2}}^{-1}\in M
$$
and so 
$
x_{\frac{n+3}{2}}^{-1} =h x_{\frac{n+3}{2}} h\in M
$. 
Thus,  
$$
x_{2i+1} = x_{2i+3}^{-1}x_{2i+5}^{-1}\cdots x_{\frac{n+1}{2}}^{-1} (x_{\frac{n+1}{2}}x_{\frac{n-3}{2}}\cdots x_5x_3)x_3^{-1}x_5^{-1}\cdots x_{2i-1}^{-1}\in M
$$
and 
$$
x_{2i+2} = x_{2i+4}^{-1}x_{2i+6}^{-1}\cdots x_{\frac{n+3}{2}}^{-1} (x_{\frac{n+3}{2}}x_{\frac{n-1}{2}}\cdots x_6x_4)x_4^{-1}x_6^{-1}\cdots x_{2i}^{-1} \in M  
$$
for $2\leqslant i\leqslant \frac{n-1}{4}$, i.e. $x_i\in M$, for $5\leqslant i\leqslant\frac{n+3}{2}$. 
Moreover, 
$$
x_{3} = x_{5}^{-1}x_{7}^{-1}\cdots x_{\frac{n+1}{2}}^{-1} (x_{\frac{n+1}{2}}x_{\frac{n-3}{2}}\cdots x_5x_3)\in M
\quad\text{and}\quad
x_{4} = x_{6}^{-1}x_{8}^{-1}\cdots x_{\frac{n+3}{2}}^{-1} (x_{\frac{n+3}{2}}x_{\frac{n-1}{2}}\cdots x_6x_4) \in M. 
$$
At last,
$$
x_1=hx_1^{-1}h=hx_nx_{n-2}\cdots x_5x_3h \in M\quad\text{and}\quad x_2=hx_2^{-1}h=hx_{n-1}x_{n-3}\cdots x_6x_4h\in M, 
$$
as required.
\end{proof}

\smallskip 

Next, consider $n\equiv2\mod 4$. Hence, $n$ is even and $\frac{n}{2}$ is odd. 

Let $h_i=\stackrel{\leftarrow}{\id}_{X_i}$, for $1\leqslant i\leqslant n$. 
Hence, for $1\leqslant i\leqslant n$, we get $h_i^2=\id_{X_i}$ and, as $\id_{X_i}\in\AM_n$ then, by Corollary \ref{AM12}, we get $h_i\in\AM_n$.  
Moreover, we have: 

\begin{lemma}\label{gen2}
If  $n\equiv2\mod 4$, then $\AM_n=\langle \AO_n,h_1,\ldots,h_n\rangle$. 
\end{lemma}
\begin{proof}
Let $\alpha\in \AM_n\setminus\AO_n$. Then, $\alpha$ has rank less than or equal to $n-1$ (as $h\notin\AM_n$) 
and so there exists $1\leqslant i\leqslant n$ such that $\dom(\alpha)\subseteq X_i$. 
Hence, $\alpha=\id_{X_i}\alpha=h_i(h_i\alpha)\in \langle \AO_n,h_1,\ldots,h_n\rangle$, as required. 
\end{proof}

\begin{theorem}\label{genAM2}
If $n\equiv2\mod 4$, then $\{ x_1h_{n-1},x_2h_n,x_3,\ldots,x_n \}$ is a generating set of $\AM_n$ with minimum size.
In particular,  $\rank(\AM_n)=n$. 
\end{theorem} 
\begin{proof}
Let $M=\langle x_1h_{n-1},x_2h_n,x_3,\ldots,x_n \rangle$.  
Then, 
$$
x_1=(x_1h_{n-1})x_{n-1}x_{n-3}\cdots x_5x_3(x_1h_{n-1}), \quad 
x_2=(x_2h_n)x_nx_{n-2}\cdots x_6x_4(x_2h_n) 
$$
and, for $1\leqslant i\leqslant\frac{n}{2}$, 
$$
h_{2i-1}=x_{2i-1}\cdots x_5x_3 (x_1h_{n-1})x_{n-1}x_{n-3}\cdots x_{2i+1} 
\quad\text{and}\quad
h_{2i}=x_{2i}\cdots x_6x_4 (x_2h_n)x_nx_{n-2}\cdots x_{2i+2},  
$$
whence $x_1,x_2,h_1,\ldots,h_n\in M$. 
Therefore, the result follows by Proposition \ref{genAO} and Lemmas \ref{gen2} and \ref{hrk23}. 
\end{proof}

\smallskip 

Finally, consider $n\equiv3\mod 4$. Observe that, $n$ is odd and $\frac{n+1}{2}$ is even.

Let $y_i=\binom{X_i}{X_{i+1}}\in\PMI_n\setminus\POI_n$, for $1\leqslant i\leqslant n-1$, and $y_n=\binom{X_n}{X_1}\in\POI_n$. 
Then, by Proposition \ref{chAO}, $y_n\in\AO_n$ (and so $y_n\in\AM_n$) and, by Proposition \ref{chAM}(1), $y_i\in\AM_n$, for $1\leqslant i\leqslant n-1$. 
Furthermore, we have: 

\begin{theorem}\label{genAM3}
If $n\equiv3\mod 4$, then $\{ y_1,\ldots,y_n \}$ is a generating set of $\AM_n$ with minimum size.
In particular,  $\rank(\AM_n)=n$. 
\end{theorem} 
\begin{proof}
Let $M=\langle y_1,\ldots,y_n \rangle$.  
First, we show that $\AO_n\subseteq M$. 
In fact, $x_1=y_1\cdots y_{n-1}$, $x_2=y_2\cdots y_{n-2}$,  $x_3=y_3\cdots y_{n}$ and, for $4\leqslant i\leqslant n$, 
$x_i=y_i\cdots y_{n-1}y_ny_1\cdots y_{i-3}$, whence $x_1,\ldots, x_n\in M$ and so, 
by Proposition \ref{genAO}, $\AO_n\subseteq M$. 

Next, we show that all elements of $\AM_n\setminus\AO_n$ with rank $n-1$ belong to $M$. 
So, let $\alpha$ be such an element. 
Then, $\alpha=\binom{X_i}{X_j}\in\PMI_n\setminus\POI_n$, for some $1\leqslant i,j\leqslant n$. Hence,  
if $j=1$, then $\alpha=y_i\cdots y_{n-1}y_n\in M$ and, if $2\leqslant j\leqslant n$, then 
$\alpha=y_i\cdots y_{n-1}y_ny_1\cdots y_{j-1}\in M$. 

In order to show that also the elements of $\AM_n\setminus\AO_n$ with rank less than $n-1$ belong to $M$, 
we begin to show that $\stackrel{\leftarrow}{\id}_{\{1,\ldots,k\}}\in M$, for $2\leqslant k\leqslant n-2$.  
So, let $2\leqslant k\leqslant n-2$. Take 
$$
\xi_1=\transf{1&2&\cdots & k-1 &k\\ n-k & n-k+1 & \cdots& n-2& n-1}
\quad\text{and}\quad 
\xi_2=\transf{1&2&3& \cdots & n-3& n-2 &n-1\\ n & n-2 & n-3 & \cdots& 3& 2& 1}. 
$$
Then, $\stackrel{\leftarrow}{\id}_{\{1,\ldots,k\}}=\xi_1\xi_2$ and, 
as $\xi_1\in\AO_n$ and $\xi_2\in\AM_n$, from what has already been established, we have $\xi_1,\xi_2\in M$, 
whence $\stackrel{\leftarrow}{\id}_{\{1,\ldots,k\}}\in M$. 
Now, let $\alpha$ be an element of $\AM_n\setminus\AO_n$ with rank $k$. Then, 
take $\alpha_1=\binom{\dom(\alpha)}{\{1,\ldots,k\}},\alpha_2=\binom{\{1,\ldots,k\}}{\im(\alpha)}\in\POI_n$. 
Then, $\alpha_1,\alpha_2\in \AO_n\subseteq M$ and so $\alpha=\alpha_1 \stackrel{\leftarrow}{\id}_{\{1,\ldots,k\}}\alpha_2\in M$.  
Now, since all elements of $\AM_n\setminus\AO_n$ have rank greater than or equal to two, we conclude that $\AM_n\setminus\AO_n\subseteq M$ 
and so $M=\AM_n$. 

Finally, by Lemma \ref{hrk23}, it follows that $\{ y_1,\ldots,y_n \}$ is a generating set of $\AM_n$ with minimum size and so, $\AM_n$ has rank $n$, as required. 
\end{proof}

\section*{Acknowledgment}

This work is funded by national funds through the FCT - Funda\c c\~ao para a Ci\^encia e a Tecnologia, I.P., 
under the scope of the projects UID/00297/2025 (https://doi.org/10.54499/UID/00297/2025) and 
UID/PRR/00297/2025 (https://doi.org/10.54499/UID/PRR/00297/2025) (Center for Mathematics and Applications \-- NOVA Math).  

\smallskip 

The author would like to thank the anonymous referees for their comments and suggestions.

\section*{Declarations} 

The author declares no conflicts of interest.

\end{document}